\documentclass[a4paper,reqno,10pt]{amsart}

\usepackage[dvipsnames]{xcolor}
\usepackage[colorlinks=true, pdfstartview=FitV, linkcolor=purple, citecolor=purple, urlcolor=purple]{hyperref}
\usepackage{amsmath}
\usepackage{amssymb}
\usepackage{amsthm}
\usepackage{bbm}
\usepackage{txfonts}
\usepackage{microtype}
\usepackage[utf8]{inputenc}
\usepackage{caption}
\usepackage{subcaption}
\usepackage{float}
\usepackage{tikz}
\usepackage{cleveref}
\usepackage[foot]{amsaddr}
\usepackage{enumitem}
\usepackage{fullpage}
\usepackage{mathrsfs}
\usepackage[percent]{overpic}
    \usetikzlibrary{positioning}
\usepackage[font=footnotesize]{caption}
\usepackage{enumitem}
\usepackage[all]{xy}
\usepackage{mathtools}

\newtheorem{theorem}{Theorem}[section]
\newtheorem{lemma}[theorem]{Lemma}
\newtheorem{proposition}[theorem]{Proposition}
\newtheorem{corollary}[theorem]{Corollary}
\newtheorem{definition}[theorem]{Definition}

\title{Intermediate $\beta$-shifts as greedy $\beta$-shifts with a hole}

\author{Niels Langeveld}
\address[Niels Langeveld]{Lehrstuhl f\"{u}r Mathematik und Statistik, Montanuniversit\"{a}t Leoben, Leoben, Austria}

\author{Tony Samuel}
\address[Tony Samuel]{School of Mathematics, University of Birmingham, Birmingham, UK}

\begin{document}

\null

\vspace{-1.25em}

\maketitle

\vspace{-2.75em}
 
\begin{abstract}
We show that every intermediate \mbox{$\beta$-transformation} is topologically conjugate to a greedy \mbox{$\beta$-transformation} with a hole at zero,  and provide a counterexample illustrating that the correspondence is not \mbox{one-to-one}.  This characterisation is employed to (1) build a Krieger embedding theorem for intermediate \mbox{$\beta$-transformation}, complementing the result of Li, Sahlsten, Samuel and Steiner [2019], and (2) obtain new metric and topological results on survivor sets of intermediate \mbox{$\beta$-transformations} with a hole at zero, extending the work of Kalle, Kong, Langeveld and Li [2020]. Further, we derive a method to calculate the Hausdorff dimension of such survivor sets as well as results on certain bifurcation sets.  Moreover, by taking unions of survivor sets of intermediate \mbox{$\beta$-transformations} one obtains an important class of sets arising in metric number theory, namely sets of badly approximable numbers in non-integer bases.  We prove, under the assumption that the underlying symbolic space is of finite type, that these sets of badly approximable numbers are winning in the sense of Schmidt games, and hence have the countable intersection property, extending the results of Hu and Yu [2014], Tseng [2009] and F\"{a}rm, Persson and Schmeling [2010].
\end{abstract}

\section{Introduction}\label{sec:intro}

The simplest and most widely studied class of expanding interval maps, and those which we will concern ourselves, are intermediate $\beta$-transformations, namely transformations of the form $T_{\beta, \alpha} : x \mapsto \beta x + \alpha\mod 1$ acting on $[0,1]$, where $(\beta, \alpha) \in \Delta \coloneqq \{ (b, a) \in \mathbb{R}^{2} \colon b \in (1, 2) \; \text{and} \; a \in [0, 2-b]\}$. This class of transformations have motivated a wealth of results, providing practical solutions to a variety of problems.  They arise as Poincar\'e maps of the geometric model of Lorenz differential equations~\cite{MR681294}, Daubechies \textsl{et al.} \cite{1011470} proposed a new approach to analog-to-digital conversion using \mbox{$\beta$-transformations}, and Jitsumatsu and Matsumura \cite{Jitsumatsu2016AT} developed a random number generator using \mbox{$\beta$-transformations}.  (This random number generator passed the NIST statistical test suite.)  Through their study, many new phenomena have appeared, revealing rich combinatorial and topological structures, and unexpected connections to probability theory, ergodic theory and aperiodic order; see for instance \cite{bezuglyi_kolyada_2003,Komornik:2011,ArneThesis}.

Intermediate $\beta$-transformations also have an intimate link to metric number theory in that they give rise to \mbox{non-integer} based expansions of real numbers. Given $\beta \in (1, 2)$ and $x \in [0, 1/(\beta-1)]$, an infinite word $\omega = \omega_{1}\omega_{2}\cdots$ with letters in the alphabet $\{0, 1\}$ is called a \textsl{$\beta$-expansion} of $x$ if
    \begin{align*}
        x = \sum_{n \in \mathbb{N}} \omega_{n} \, \beta^{-n}.
    \end{align*}
Through iterating the map $T_{\beta, \alpha}$ one obtains a subset of $\{ 0, 1\}^{\mathbb{N}}$ known as the intermediate $\beta$-shift $\Omega_{\beta, \alpha}$, where each $\omega \in \Omega_{\beta, \alpha}$ is a $\beta$-expansion, and corresponds to a unique point in $[\alpha/(\beta-1), 1+\alpha/(\beta-1)]$, see \eqref{eq:commutative_diag} and the commentary following it for further details. By equipping $\Omega_{\beta, \alpha}$ with the left shift map $\sigma$, one obtains a dynamical system which is topologically conjugate to the dynamical system $\mathcal{S}_{\beta,\alpha}=(T_{\beta, \alpha}, [0, 1])$, namely one obtains a symbolic system which possess the same ergodic properties as $\mathcal{S}_{\beta,\alpha}$. Note, in this article, by topologically conjugate we mean that the conjugacy is one-to-one everywhere except on a countable set on which the conjugacy is at most finite to one.

Open dynamical systems, namely systems with holes in the state space through which mass can leak away, have received a lot of attention, see \cite{Ur86,Schme,N09,BBF2014,KKLL} and reference therein.  We prove the following correspondence connecting $\mathcal{S}_{\beta,\alpha}$ and open dynamical systems driven by greedy $\beta$-transformations, namely intermediate $\beta$-transformations with no rotation factor, or equivalently when $\alpha = 0$.

    \begin{theorem}\label{thm:main}
        Given $(\beta, \alpha) \in \Delta$, there exist $t \in [0, 1]$ and $\beta' \in (1, 2)$ with $(T_{\beta, \alpha}, [0, 1])$ topologically conjugate to the open dynamical system $(T_{\beta', 0}\vert_{K^{+}_{\beta',0}(t)}, K^{+}_{\beta',0}(t))$, where 
            \begin{align*}
                K^{+}_{\beta', 0}(t) \coloneqq \{ x \in[0, 1) \colon T_{\beta',0}^{n}(x)\not \in [0,t) \; \textup{for all} \; n \in \mathbb{N}_{0} \}.
            \end{align*}
        However, the converse does not hold, namely there exist $t \in [0, 1]$ and $\beta' \in (1, 2)$ such that there does not exist a topological conjugation between $(T_{\beta', 0}\vert_{K^{+}_{\beta',0}(t)}, K^{+}_{\beta', 0}(t))$ and $(T_{\beta, \alpha}, [0, 1])$ for any $(\beta, \alpha) \in \Delta$. Moreover, given $\beta' \in (1, 2)$ with $T_{\beta',0}^{n}(1) = 0$, for some $n \in \mathbb{N}$, there exists $\delta \in (0, \beta'^{-1})$, such that to each $t < \delta$ in the bifurcation set 
                \begin{align*}
                    E_{\beta',0}^{+} \coloneqq \{ t \in[0,1) \colon T_{\beta', 0}^{n}(t) \not\in [0, t) \; \textup{for all} \; n \in \mathbb{N}_{0} \}
                \end{align*}
        one may associate a unique $(\beta, \alpha) \in \Delta$ with $(T_{\beta, \alpha}, [0, 1])$ topologically conjugate to $(T_{\beta', 0}\vert_{K^{+}_{\beta',0}(t)}, K^{+}_{\beta',0}(t))$.
    \end{theorem}

This result complements \cite[Proposition 3.1 and Theorem 3.5]{bundfuss_kruger_troubetzkoy_2011}. Here, it is shown that every subshift of finite type and any greedy $\beta$-shift encodes a survivor set of $x \mapsto mx \bmod 1$, for some $m \in \mathbb{N}$ with $m \geq 2$.  With this and \Cref{thm:main} at hand, we have that any intermediate $\beta$-shift encodes a survivor set of the doubling map.

We employ our characterisation given in \Cref{thm:main} to (1) build a Krieger embedding theorem for intermediate \mbox{$\beta$-transformations}, and (2) obtain new metric and topological results on survivor sets of intermediate \mbox{$\beta$-transformations}.
    \begin{enumerate}[label={\rm(\arabic*)},leftmargin=*]
        \item {\bfseries A Krieger embedding theorem for intermediate \mbox{$\beta$-transformations.}} Subshifts, such as $\Omega_{\beta, \alpha}$, are to dynamical systems what shapes like polygons and curves are to geometry. Subshifts which can be described by a finite set of forbidden words are called \textsl{subshifts of finite type} and play an essential role in the study of dynamical systems. One reason why subshifts of finite type are so useful is that they have a simple representation using a finite directed graph. Questions concerning the subshift can then often be phrased as questions about the graph’s adjacency matrix, making them more tangible, see for instance \cite{LM,brin_stuck_2002} for further details on subshifts of finite type. Moreover, in the case of greedy $\beta$-shifts (that is when $\alpha = 0$), often one first derives results for greedy $\beta$-shifts of finite type, and then one uses an approximation argument to determine the result for a general greedy $\beta$-shift, see for example \cite{DavidFarm2010,LL16}.  Here we prove a Krieger embedding theorem for intermediate $\beta$-shifts.  Namely, we show the following, complementing the work of \cite{LSSS} where the same result is proven except where the containment property given in Part~(iii) is reversed.  Due to this reversed containment, our proof and that of \cite{LSSS}, although both of a combinatorial flavour, are substantially different.
            \begin{corollary}\label{Cor_1}
                Given $(\beta, \alpha) \in \Delta$, there exists a sequence $\{ (\beta_{n}, \alpha_{n}) \}_{n \in \mathbb{N}}$ in $\Delta$ with $\lim_{n\to \infty} (\beta_{n}, \alpha_{n}) = (\beta, \alpha)$ and
                    \begin{enumerate}[label={\rm(\roman*)}]
                        \item $\Omega_{\beta_{n}, \alpha_{n}}$ a subshift of finite type,
                        \item the Hausdorff distance between $\Omega_{\beta, \alpha}$ and $\Omega_{\beta_{n}, \alpha_{n}}$ converges to zero as $n$ tends to infinity, and 
                        \item $\Omega_{\beta_{n}, \alpha_{n}} \subseteq \Omega_{\beta, \alpha}$.
                    \end{enumerate}
            \end{corollary}
        \noindent These results together with the results of \cite{LSSS} complements the corresponding result for the case when $\alpha = 0$ proven in \cite{P1960} and which asserts that any greedy $\beta$-shift can be approximated from \textsl{above} and \textsl{below} by a greedy $\beta$-shift of finite type.
        \vspace{1em}
        \item {\bfseries Metric and topological results on survivor sets of intermediate \mbox{$\beta$-transformations}.} Via our correspondence theorem (\Cref{thm:main}), we are able to transfer the results of \cite{KKLL} obtained for open dynamical systems driven by greedy $\beta$-transformations to general intermediate $\beta$-transformations.  Specifically, we show the following, extending the results of \cite{KKLL} and complementing those of \cite{Ur86,N09}. Here we follow the notation used in \cite{KKLL}, and recall  that an infinite word in the alphabet $\{0, 1\}$ is \textsl{balanced} if and only if the number of ones in any two subwords of the same length differ by at most $1$.
            \begin{corollary}\label{Cor_2}
            The bifurcation set $E_{\beta,\alpha}^{+} \coloneqq \{ t \in[0,1) \colon T_{\beta, \alpha}^{n}(t) \not\in [0, t) \; \textup{for all} \; n \in \mathbb{N}_{0} \}$ is a Lebesgue null set. Moreover, if largest lexicographic word in $\Omega_{\beta, \alpha}$ is balanced, then $E_{\beta,\alpha}^{+}$ contains no isolated points.
            \end{corollary}
        \noindent If the largest lexicographic word in $\Omega_{\beta, \alpha}$ is not balanced, then under an additional technical assumption, in \Cref{cor:isoloated_pts}, we show that there exists a $\delta > 0$, such that $E_{\beta,\alpha}^{+} \cap [0, \delta]$ contains no isolated points. Further, letting  $K^{+}_{\beta, \alpha}(t)$ denote the survivor set $\{ x \in[0,1) \colon T_{\beta,\alpha}^{n}(x)\not \in [0,t) \; \text{for all} \; n \in \mathbb{N}_{0} \}$, we have:
            \begin{corollary}\label{Cor_3}
            The dimension function $\eta_{\beta, \alpha} \colon t \mapsto \dim_{\mathcal{H}}(K_{\beta,\alpha}^{+}(t))$ is a Devil staircase function, that is, $\eta_{\beta,\alpha}(0) = 1$, $\eta_{\beta,\alpha}((1-\alpha)/\beta) = 0$, $\eta_{\beta, \alpha}$ is decreasing, and $\eta_{\beta,\alpha}$ is constant Lebesgue almost everywhere.
            \end{corollary}
        \noindent With \Cref{Cor_1,Cor_3} at hand, we can also prove the following.
            \begin{corollary}\label{Cor_E_beta_alpha}
                The bifurcation set $E_{\beta,\alpha}^{+}$ has full Hausdorff dimension.
            \end{corollary}
\end{enumerate}

The sets $K^{+}_{\beta,\alpha}(t)$ can be seen as a level sets of the set of badly approximable numbers in non-integer bases, that is,
    \begin{align*}
        \mathrm{BAD}_{\beta, \alpha}(0) \coloneqq \{ x \in [0,1] \colon 0 \not\in \overline{\{T_{\beta, \alpha}^n(x) \colon n\geq 0\}}\} = \bigcup_{t \in (0, 1)} K^{+}_{\beta, \alpha}(t).
    \end{align*}
Moreover, for $\xi \in [0,1]$ one can study the more general set
    \begin{align*}
        \mathrm{BAD}_{\beta, \alpha}(\xi) \coloneqq \{ x \in [0,1] \colon \xi \not\in \overline{\{T_{\beta, \alpha}^n(x) \colon n\geq 0\}}\},
    \end{align*}
which, by \Cref{Cor_3}, is a set of full Hausdorff dimension.
    
When $\alpha = 0$, F\"arm, Persson and Schmeling \cite{DavidFarm2010} and later Hu and Yu \cite{HY} study these sets and showed that they are winning, and hence that they have the large intersection property.  To our knowledge, the present work, is the first to consider the case $\alpha \neq 0$.  Before stating our results on $\mathrm{BAD}_{\beta, \alpha}(\xi)$, we recall the notion of a winning set.

In the 1960s Schmidt~\cite{S} introduced a topological game in which two players take turns in choosing balls that are a subset of the previously chosen ball. There is a target set $S$ and the objective of Player~$1$ is to make sure that the point that is present in every ball chosen during the game is in $S$. The objective of Player~$2$ is to prevent this. A set is called winning when Player~$1$ can always build a winning strategy no matter how Player~2 plays.

    \begin{definition}
        Let $\alpha$ and $\gamma\in (0,1)$ be fixed and suppose we have two players, Player~1 and Player~2.  Let Player~2 choose a closed initial interval $B_1\subset [0,1]$ and let Player~1 and Player~2 choose nested closed intervals such that $B_1 \supset W_1 \supset B_2 \supset W_2 \supset \ldots$ and $|W_{n+1}|=\alpha |B_n|$ and  $|B_{n+1}|=\gamma |W_n|$. A set $S$ is called $(\alpha,\gamma)$-winning if there is a strategy for Player~2 to ensure that $\bigcap_{i\in \mathbb{N}} W_i \subset S$. The set $S$ is called $\alpha$-winning if it is $(\alpha,\gamma)$-winning for all $\gamma\in (0,1)$ and is called winning if it is $\alpha$-winning for some $\alpha\in (0,1)$.
    \end{definition}

A key attribute of winning which makes it an interesting property to study is that winning sets have full Hausdorff dimension~\cite{S}. Another, is that it persists under taking intersections, that is, for two winning sets their intersection is again winning, and hence of full Hausdorff dimension~\cite{S}; this is not true in general for sets of full Hausdorff dimension.  We also note, the property of winning is preserved under bijective affine transformations.

    \begin{theorem}\label{thm:main_2}
        Given $(\beta, \alpha) \in \Delta$ with $\Omega_{\beta, \alpha}$ a subshift of finite type, and $\xi \in [0, 1]$, the set $\mathrm{Bad}_{\beta,\alpha}(\xi)$ is winning.
    \end{theorem}

We remark that in \cite{JT2009} a similar result for $C^{2}$-expanding Markov circle maps was proven, but that intermediate $\beta$-transformations do not fall into this regime. Further, with \Cref{thm:main_2} at hand and with \cite[Theorem 1]{DavidFarm2010} in mind, we conjecture that $\mathrm{Bad}_{\beta,\alpha}(\xi)$ is winning for all $(\beta, \alpha) \in \Delta$ and $\xi \in [0, 1]$.

Our work is organised as follows. In \Cref{sec:prelim} we we present necessary definitions, preliminaries and auxiliary results. \Cref{sec:proof_thm_1_1,sec:proof_thm_1_6} are respectively devoted to proving \Cref{thm:main,thm:main_2}, and \Cref{sec:proof_cor_1_2,sec:proof_cor_1_3_4} respectively contain the proofs of \Cref{Cor_1}, and \Cref{Cor_2,Cor_3}. Additionally, in \Cref{sec:proof_cor_1_3_4}, we demonstrate how our theory may be used to numerically compute the Hausdorff dimension of $K_{\beta,\alpha}^{+}(t)$.

\section{Notation and preliminaries}\label{sec:prelim}
\subsection{Subshifts}

Let $m \geq 2$ denote a natural number and set $\Lambda = \{0, 1, \ldots, m-1\}$. We equip the space $\Lambda^\mathbb{N}$ of infinite sequences indexed by $\mathbb{N}$ with the topology induced by the \textsl{word metric} $\mathscr{D} \colon \Lambda^\mathbb{N} \times \Lambda^\mathbb{N} \to \mathbb{R}$ given by
    \begin{align*}
    \mathscr{D}(\omega, \nu) \coloneqq
        \begin{cases}
        0 & \text{if} \; \omega = \nu,\\
        2^{- \lvert\omega \wedge \nu\rvert + 1} & \text{otherwise}.
        \end{cases}
    \end{align*}
Here, $\rvert \omega \wedge \nu \lvert \coloneqq \min \, \{ \, n \in \mathbb{N} \colon \omega_{n} \neq \nu_n \}$, for $\omega$ and $\nu \in \Lambda^{\mathbb{N}}$ with $\omega \neq \nu$, where for an element $\omega \in \Lambda^{\mathbb{N}}$ we write $\omega=\omega_1\omega_2\cdots$. Note, when equipping $\Lambda$ with the discrete topology, the topology induced by $\mathscr{D}$ on $\Lambda^{\mathbb{N}}$ coincides with the product topology on $\Lambda^{\mathbb{N}}$. We let $\sigma \colon \Lambda^{\mathbb{N}} \to \Lambda^{\mathbb{N}}$ denote the \textsl{left-shift map} defined by $\sigma(\omega_{1} \omega_{2} \cdots) \coloneqq \omega_{2} \omega_{3} \cdots$, and for $n \in \mathbb{N}$, we set $\omega\rvert_{n} = \omega_{1} \omega_{2} \cdots \omega_{n}$. A \textsl{subshift} is any closed set $\Omega \subseteq \Lambda^\mathbb{N}$ with $\sigma(\Omega) \subseteq \Omega$.  Given a subshift $\Omega$, we set $\Omega\vert_{0} = \{ \varepsilon\}$, where $\varepsilon$ denotes the empty word, and for $n \in \mathbb{N}$, we set
    \begin{align*}
    \Omega\lvert_{n} \coloneqq \left\{ \omega_{1} \cdots \omega_{n} \in \Lambda^{n} \colon \,\text{there exists} \; \xi  \in \Omega \; \text{with} \; \xi|_n = \omega_{1} \cdots \omega_{n} \right\}
    \end{align*}
and write $\Omega^{*} \coloneqq \bigcup_{n \in \mathbb{N}_{0}} \Omega\lvert_{n}$ for the collection of all finite words.  We denote by $\lvert \Omega\vert_{n} \rvert$ the cardinality of $\Omega\vert_{n}$, and for $\omega \in \Omega\vert_{n}$, we set $\lvert \omega \rvert = n$.  We extend the domain of $\sigma$ to $\Omega^{*}$, by setting $\sigma(\varepsilon) \coloneqq \varepsilon$, and for $n \in \mathbb{N}$, letting 
    \begin{align*}
        \sigma(\omega_{1} \omega_{2} \cdots \omega_{n}) \coloneqq 
            \begin{cases}
                \omega_{2} \omega_{3} \cdots \omega_{n} & \text{if} \; n \neq 1,\\
                \varepsilon & \text{otherwise}.
            \end{cases}
    \end{align*}
For $\omega = \omega_{1} \cdots \omega_{\lvert \omega \rvert} \in \Omega^{*}$ and $\xi = \xi_{1} \xi_{2} \cdots \in \Omega \cup \Omega^{*}$ we denote the concatenation $\omega_{1} \cdots \omega_{\lvert \omega \rvert} \ \xi_{1} \xi_{2} \cdots$ by $\omega \ \xi$.

    \begin{definition}
        A subshift $\Omega$ is said to be \textsl{of finite type} if there exists $M \in \mathbb{N}$ such that, $\omega_{n - M + 1} \cdots \omega_{n} \ \xi_{1} \cdots \xi_{m} \in \Omega^{*}$, for all $\omega_{1} \cdots \omega_{n}$ and $\xi_{1} \cdots \xi_{m} \in \Omega^{*}$ with $n, m \in \mathbb{N}$ and $n \geq M$, if and only if $\omega_{1} \cdots \omega_{n} \ \xi_{1} \cdots \xi_{m} \in \Omega^{*}$. 
    \end{definition}
 
The following result gives an equivalent condition for when a subshift is of finite type.

    \begin{theorem}[{\cite[Theorem 2.1.8]{LM}}]
        A subshift $\Omega \subseteq \Lambda^{\mathbb{N}}$ is of finite type if and only if there exists a finite set $F \subset \Omega^{*}$ with $\Omega = \mathcal{X}_{F}$, where $\mathcal{X}_{F} \coloneqq \{ \omega \in \Lambda^{\mathbb{N}} \colon \sigma^{m}(\omega)\vert_{\lvert \xi \rvert} \neq \xi \; \text{for all} \; \xi \in F \; \text{and} \; m \in \mathbb{N}\}$.
    \end{theorem}
 
Two subshifts $\Omega$ and $\Psi$ are said to be \textsl{topologically conjugate} if there exists a $\phi \colon \Omega \to \Psi$ that is surjective, one-to-one everywhere except on a countable set on which it is at most finite-to-one, and $\sigma \circ \phi(\omega) = \phi \circ \sigma(\omega)$ for all $\omega \in \Omega$.  We call $\phi$ the \textsl{conjugacy}. In the case that $m = 2$, a particular conjugacy which we will make use of is the \textsl{reflection map} $R$ defined by $R(\omega_{1} \omega_{2} \cdots) = (1-\omega_{1})(1-\omega_{2})\cdots$ for $\omega = \omega_{1} \omega_{2} \cdots \in \{0,1\}^{\mathbb{N}}$. This concept of two subshifts being topologically conjugate, naturally extends to general dynamical systems, see for instance \cite{LM,brin_stuck_2002}.

An infinite word $\omega = \omega_{1} \omega_{2} \cdots \in \Lambda^{\mathbb{N}}$ is called \textsl{periodic} with \textsl{period} $n \in \mathbb{N}$ if and only if, for all $m \in \mathbb{N}$, we have $\omega_{1} \cdots \omega_{n} = \omega_{(m - 1)n + 1} \cdots \omega_{m n}$, in which case we write $\omega = \omega\vert_{n}^{\infty}$, and denote the smallest period of $\omega$ by $\operatorname{per}(\omega)$.  Similarly, an infinite word $\omega = \omega_{1} \omega_{2} \cdots \in \Lambda^{\mathbb{N}}$ is called \textsl{eventually periodic} with \textsl{period} $n \in \mathbb{N}$ if there exists $k \in \mathbb{N}$ such that, for all $m \in \mathbb{N}$, we have $\omega_{k+1} \cdots \omega_{k+n} = \omega_{k+(m - 1)n + 1} \cdots \omega_{k+ m n}$, in which case we write $\omega = \omega_{1} \cdots \omega_{k} (\omega_{k+1} \cdots \omega_{k+n})^\infty$.

\subsection{Intermediate \texorpdfstring{$\beta$}{beta}-shifts}\label{sec:beta-shifts}

For $(\beta, \alpha) \in \Delta$ we set $p = p_{\beta, \alpha} = (1-\alpha)/\beta$ and define the \textsl{upper $T_{\beta, \alpha}$-expansion} $\tau_{\beta, \alpha}^{+}(x)$ of $x \in [0, 1]$ to be the infinite word $\omega_{1} \omega_{2} \cdots  \in \{ 0, 1\}^{\mathbb{N}}$, where, for $n \in \mathbb{N}$,
    \begin{align}\label{eq:upper_kneading}
        \omega_{n} \coloneqq
            \begin{cases}
                0 & \quad \text{if } T_{\beta,\alpha}^{n-1}(x) < p,\\
                1 & \quad \text{otherwise,}
            \end{cases}
    \end{align}
and define the \textsl{lower $T_{\beta, \alpha}$-expansion} $x$ to be $\tau^{-}_{\beta, \alpha}(x) \coloneqq \lim_{y \nearrow x} \tau_{\beta,\alpha}^{+}(y)$. Note, one can also define $\tau^{-}_{\beta, \alpha}(x)$ analogously to $\tau^{+}_{\beta, \alpha}(x)$ by using the map $T_{\beta,\alpha}^{-} \colon x \mapsto \beta x + \alpha$ if $x \leq p$, and $x \mapsto \beta x + \alpha - 1$ otherwise, in replace of of $T_{\beta, \alpha}$, and by changing the \textsl{less than}, to \textsl{less than or equal to} in \eqref{eq:upper_kneading}, see \cite[Section 2.2]{LSSS}.  With this in mind, and for ease of notation, sometimes we may write $T_{\beta,\alpha}^{+}$ for $T_{\beta, \alpha}$.

We denote the images of $[0,1)$ under $\tau_{\beta, \alpha}^{+}$ by $\Omega^{+}_{\beta, \alpha}$, the image of $(0,1]$ under $\tau_{\beta, \alpha}^{-}$ by $\Omega^{-}_{\beta, \alpha}$, and set $\Omega_{\beta, \alpha} \coloneqq \Omega_{\beta, \alpha}^{+} \cup \Omega_{\beta, \alpha}^{-}$.  We refer to $\Omega_{\beta, \alpha}$ as an intermediate $\beta$-shift and define the \textsl{upper} and \textsl{lower kneading invariants} of $\Omega_{\beta,\alpha}$ to be the infinite words $\tau^{\pm}_{\beta, \alpha}(p)$, respectively. The following result shows that $\tau^{\pm}_{\beta, \alpha}(p)$ completely determine $\Omega_{\beta,\alpha}$.

    \begin{theorem}[{\cite{P1960,HS:1990,AM:1996,KS:2012,BHV:2011}}]\label{thm:Structure}
        For $(\beta, \alpha) \in \Delta$, the spaces $\Omega_{\beta, \alpha}^{\pm}$ are completely determined by upper and lower kneading invariants of $\Omega_{\beta, \alpha}$, namely
            \begin{align*}
                \Omega_{\beta, \alpha}^{+} &= \{ \omega \in \{ 0, 1\}^{\mathbb{N}} \colon \tau_{\beta, \alpha}^{+}(0) \preceq \sigma^{n}(\omega) \prec \tau_{\beta, \alpha}^{-}(p) \; \textup{or} \; \tau_{\beta, \alpha}^{+}(p) \preceq \sigma^{n}(\omega) \prec \tau_{\beta, \alpha}^{-}(1) \; \textup{for all} \; n \in \mathbb{N}_{0} \},\\
                \Omega_{\beta, \alpha}^{-} &= \{ \omega \in \{ 0, 1\}^{\mathbb{N}} \colon \tau_{\beta, \alpha}^{+}(0) \prec  \sigma^{n}(\omega) \preceq \tau_{\beta, \alpha}^{-}(p) \; \textup{or} \; \tau_{\beta, \alpha}^{+}(p) \prec \sigma^{n}(\omega) \preceq \tau_{\beta, \alpha}^{-}(1) \; \textup{for all} \; n \in \mathbb{N}_{0} \}.
            \end{align*}
        Here, $\prec$, $\preceq$, $\succ$ and $\succeq$ denote the lexicographic orderings on $\{ 0 ,1\}^{\mathbb{N}}$.  Moreover, the cardinality of $\Omega_{\beta, \alpha}^{\pm}$ is equal to that of the continuum, and $\Omega_{\beta, \alpha}$ is closed with respect to the metric $\mathscr{D}$. Hence, $\Omega_{\beta, \alpha}$ is a subshift.
    \end{theorem}

This result establishes the importance of the kneading invariants of $\Omega_{\beta, \alpha}$, for a given $(\beta, \alpha) \in \Delta$, and so it is natural to ask, for a fixed $\beta \in (1, 2)$, if they are monotonic or continuous in $\alpha$.  The following proposition answers this.

    \begin{proposition}[{\cite{BHV:2011,Cooperband2018ContinuityOE}}]\label{prop:mon_cont_kneading}
        Let $\beta \in (1,2)$ be fixed.
            \begin{enumerate}[label={\rm(\arabic*)}]
                \item The maps $a \mapsto \tau_{\beta,a}^{\pm}(p_{\beta, a})$ are strictly increasing with respect to the lexicographic ordering.
                \item The map $a \mapsto \tau_{\beta, a}^{+}(p_{\beta, a})$ is right continuous, and the map $a \mapsto \tau_{\beta,a}^{-}(p_{\beta, a})$ is left continuous.
                \item If $\alpha \neq 0$ and $\tau_{\beta, \alpha}^{+}(p_{\beta, \alpha})$ is not periodic, then $a \mapsto \tau_{\beta,a}^{+}(p_{\beta, a})$ is continuous at $\alpha$, and if $\alpha \neq \beta-1$ and $\tau_{\beta,\alpha}^{-}(p_{\beta, \alpha})$ is not periodic, then $a \mapsto \tau_{\beta,a}^{-}(p_{\beta, a})$ is continuous at $\alpha$.
                \item If $\tau_{\beta,\alpha}^{+}(p_{\beta, \alpha})$ is periodic with period $M$, for a given $\alpha \in (0, 2-\beta]$, then given $m \in \mathbb{N}$, there exists a real number $\delta > 0$ so that, $\tau_{\beta,\alpha-\delta'}^{+}(p_{\beta, \alpha-\delta'})\vert_{m} = \tau_{\beta,\alpha}^{+}(p_{\beta, \alpha})\vert_{M}\tau_{\beta,\alpha}^{-}(p_{\beta, \alpha})\vert_{m-M}$, for all $\delta' \in (0, \delta)$. 
                \item If $\tau_{\beta,\alpha}^{-}(p_{\beta, \alpha})$ is periodic with period $M$, for a given $\alpha \in [0, 2-\beta)$, then given $m \in \mathbb{N}$, there exists a real number $\delta > 0$ so that, $\tau_{\beta,\alpha+\delta'}^{-}(p_{\beta, \alpha+\delta'})\vert_{m} = \tau_{\beta,\alpha}^{-}(p_{\beta, \alpha})\vert_{M}\tau_{\beta,\alpha}^{+}(p_{\beta, \alpha})\vert_{m-M}$, for all $\delta'  \in (0, \delta)$.
            \end{enumerate}
    \end{proposition}

Another natural question to ask is when do two infinite words in the alphabet $\{0, 1\}$ give rise to kneading invariants of an intermediate $\beta$-shift. This question was addressed in \cite{barnsley_steiner_vince_2014} where the following solution was derived and for which we require the following notation. Given $\omega$ and $\nu \in \{0,1\}^\mathbb{N}$ with $\sigma(\nu) \preceq \omega \preceq \nu \preceq \sigma(\omega)$ we set
    \begin{align*}
        \Omega^{+}(\omega,\nu) &\coloneqq \{ \xi \in \{0,1\}^{\mathbb{N}} \colon \sigma(\nu) \preceq \sigma^{n}(\xi) \prec \omega \; \text{or} \; \nu \preceq \sigma^{n}(\xi) \prec \sigma(\omega) \; \text{for all} \; n \in \mathbb{N}_{0} \},\\
        \Omega^{-}(\omega,\nu) &\coloneqq \{ \xi \in \{0,1\}^{\mathbb{N}} \colon \sigma(\nu) \prec \sigma^n(\xi) \preceq \omega \; \text{or} \; \nu \prec \sigma^{n} (\xi) \preceq \sigma(\omega) \; \text{for all} \; n \in \mathbb{N}_{0} \}.
    \end{align*}

    \begin{theorem}[{\cite{barnsley_steiner_vince_2014}}]\label{thm:BSV14}
        Two infinite words $\omega = \omega_{1}\omega_{2}\cdots$ and $\nu=\nu_{1}\nu_{2}\cdots \in \{0,1\}^\mathbb{N}$ are kneading invariants of an intermediate $\beta$-shift $\Omega_{\beta, \alpha}$, for some $(\beta, \alpha) \in \Delta$ if and only if the following four conditions hold.
            \begin{enumerate}[label={\rm(\arabic*)}]
                \item $\omega_1=0$ and $\nu_1=1$, 
                \item $\omega\in\Omega^-(\omega, \nu) $ and $\nu\in\Omega^+(\omega, \nu)$,
                \item $\lim_{n\to\infty} \log(|\Omega^+|_n)>0$, and
                \item if $\omega,\nu\in\{\xi,\zeta \}^\mathbb{N}$ for two finite words $\xi$ and $\zeta$ in the alphabet $\{0,1\}$ with length greater than or equal to three, such that $\xi_1\xi_2=01$ , $\zeta_1\zeta_2=10$, $\xi^\infty \in \Omega^-(\xi^\infty, \zeta^\infty)$ and $\zeta^\infty \in \Omega^{+}(\xi^\infty, \zeta^\infty)$, then $\omega=\xi^\infty$ and $\nu=\zeta^\infty$.
            \end{enumerate}
    \end{theorem}

This result together with \Cref{thm:Structure} can be seen as a generalisation of the following seminal result of Parry.

    \begin{theorem}[{\cite[Corollary 1]{P1960}}]\label{thm:Parry_converse}
        If $\omega \in \{0, 1\}^{\mathbb{N}}$ with $\sigma^{n}(\omega) \neq 0^\infty$ for all $n \in \mathbb{N}$, then there exists a $\beta \in (1,2)$ with $\omega = \tau_{\beta, 0}^{-}(1)$ if and only if  $\sigma^{m}(\omega) \preceq \omega$ for all $m \in \mathbb{N}$.
    \end{theorem}

Combining this result with \Cref{thm:Structure}, we obtain the following which will be utilised in our proof of \Cref{thm:main}.

    \begin{corollary}\label{cor:From_greedy_to_intermediate}
        Given $(\beta, \alpha) \in \Delta$, there exists $\beta' \in (1,2)$ such that $\tau_{\beta,\alpha}^{-}(1) = \tau_{\beta',0}^{-}(1)$.
    \end{corollary}

In the sequel we will also make use of the projection $\pi_{\beta, \alpha} \colon \{ 0, 1 \}^{\mathbb{N}} \to [0, 1]$ defined by
    \begin{align*}
        \pi_{\beta, \alpha}(\omega_{1} \omega_{2} \cdots) \coloneqq \frac{\alpha}{1 - \beta} + \sum_{k \in \mathbb{N}} \frac{\omega_{k}}{\beta^k}.
    \end{align*}
We note that $\pi_{\beta, \alpha}$ is linked to the iterated function systems $( [0, 1]; f_{0} \colon x \mapsto \beta^{-1}x, f_{1} \colon x \mapsto \beta^{-1}(x+1)$ via the equality
    \begin{align*}
        \pi_{\beta, \alpha}(\omega_{1} \omega_{2} \cdots)
        = \alpha / (1-\beta) + \lim_{n \to \infty} f_{\omega_{1}} \circ \cdots \circ f_{\omega_{n}}([0, 1]),
    \end{align*}
and the iterated function system $( [0, 1]; f_{\beta, \alpha, 0} \colon x \mapsto \beta^{-1}x - \alpha\beta^{-1}, f_{\beta, \alpha, 1} \colon x \mapsto \beta^{-1}x - (\alpha-1)\beta^{-1}$ via the equality
    \begin{align}\label{eq:alt_IFS}
        \pi_{\beta, \alpha}(\omega_{1} \omega_{2} \cdots)
        =\lim_{n \to \infty} f_{\beta,\alpha,\omega_{1}} \circ \cdots \circ f_{\beta,\alpha,\omega_{n}}([0, 1]).
    \end{align}
We refer the reader to \cite{F:1990} for further details on iterated function systems. An important property of $\pi_{\beta, \alpha}$ is that the following diagrams commute.
	\begin{align}\label{eq:commutative_diag}
	    \begin{aligned}
            \xymatrix@C+2pc{
            \Omega^{+}_{\beta, \alpha}
            \ar@/_/[d]_{\pi_{\beta, \alpha}}
            \ar[r]^{\sigma} & 
            \Omega_{\beta, \alpha}^{+}
            \ar@/^/[d]^{\pi_{\beta, \alpha}} \\
            {[0, 1)}
            \ar@/_/[u]_{\tau_{\beta,\alpha}^{+}}
            \ar[r]_{T_{\beta, \alpha}}  & 
            \ar@/^/[u]^{\tau_{\beta,\alpha}^{+}}
            [0, 1)}
	    \end{aligned}
	    \qquad\qquad
	    \begin{aligned}
            \xymatrix@C+2pc{
            \Omega^{-}_{\beta, \alpha}
            \ar@/_/[d]_{\pi_{\beta, \alpha}}
            \ar[r]^{\sigma} & 
            \Omega_{\beta, \alpha}^{-}
            \ar@/^/[d]^{\pi_{\beta, \alpha}} \\
            {(0, 1]}
            \ar@/_/[u]_{\tau_{\beta,\alpha}^{-}}
            \ar[r]_{T^{-}_{\beta, \alpha}}  & 
            \ar@/^/[u]^{\tau_{\beta,\alpha}^{-}}
            (0, 1]}
	    \end{aligned}
	\end{align}

This result is verifiable from the definitions of the involved maps and a sketch of a proof can be found in \cite{BHV:2011}.  It also yields that each $\omega \in \Omega_{\beta, \alpha}$ is a $\beta$-expansion, and corresponds to the unique point 
    \begin{align*}
    \sum_{k \in \mathbb{N}} \frac{\omega_{k}}{\beta^{k}} = \pi_{\beta,\alpha}(\omega) - \frac{\alpha}{1-\beta}
    \end{align*}
in $[\alpha/(\beta-1), 1+\alpha/(\beta-1)]$. A particular expansion which we will make use of is $\tau_{\beta,0}^{-}(1)$ which is referred to as the \textsl{quasi-greedy $\beta$-expansion of $1$}.

The commutativity of the diagrams given in \eqref{eq:commutative_diag} also implies that the dynamical systems $(\Omega_{\beta, \alpha}, \sigma)$ and $([0, 1), T_{\beta, \alpha})$ are topologically conjugate.  This in tandem with \Cref{thm:Structure}, yields that the upper and lower kneading invariants completely determine the dynamics of $T_{\beta, \alpha}$. Additionally, we have the following.

    \begin{theorem}[{\cite{GH,BHV:2011}}]\label{thm:Laurent}
        Let $(\beta, \alpha) \in \Delta$ be fixed. If $\omega = \omega_{1}\omega_{2} \cdots $ and $\nu = \nu_{1} \nu_{2} \cdots $, respectively, denote the upper and lower kneading invariants of $\Omega_{\beta, \alpha}$, then $\beta$ is the maximal real root of the Laurent series 
            \begin{align*}
                \pi_{z,\alpha}(\omega) - \pi_{z,\alpha}(\nu)
                = \sum_{k \in \mathbb{N}} (\omega_{k} - \nu_{k})\,z^{-k}.
            \end{align*}
\end{theorem}

The above allows us to transfer results on the dynamical system $(\Omega_{\beta, \alpha}, \sigma)$ to $([0, 1], T_{\beta, \alpha})$ and vice versa.  We will utilise this in the proofs of our main results, and thus will make use of the following symbolic representations of $K_{\beta, \alpha}^{+}(t)$ and $E_{\beta, \alpha}^{+}$ defined in \Cref{sec:intro}.  For $(\beta, \alpha) \in \Delta$ and $t \in [0,1]$, we set
    \begin{align*}
        \mathcal{K}^{+}_{\beta,\alpha}(t) &\coloneqq \{ \omega \in \{0,1\}^\mathbb{N} \colon \tau_{\beta,\alpha}^{+}(t) \preceq \sigma^{n}(\omega) \prec \tau_{\beta,\alpha}^{-}(1) \; \text{for all} \; n \in \mathbb{N}_{0} \}
    \intertext{and we let}
        \mathcal{E}_{\beta,\alpha}^{+} &\coloneqq \{ \omega \in \{0,1\}^\mathbb{N} \colon \omega \preceq \sigma^{n}(\omega) \prec \tau_{\beta,\alpha}^{-}(1) \; \text{for all} \; n \in \mathbb{N}_{0} \}.
    \end{align*}
In addition to this, we will utilise the following Ledrappier-Young formula due to Raith \cite{R}.  For $t \in (0,1]$,
    \begin{align}\label{eq:entopen}
        \dim_{H}(K_{\beta,\alpha}(t)) = \frac{h_{\operatorname{top}}(T_{\beta,\alpha}\vert_{K_{\beta,\alpha}(t)})}{\log(\beta)}.
    \end{align}
where $K_{\beta,\alpha}(t) \coloneqq \{ x \in[0,1] \colon T_{\beta,\alpha}^{n}(x)\not \in (0,t) \; \text{for all} \; n \in \mathbb{N}_{0} \}$ and where $h_{\operatorname{top}}(T_{\beta,\alpha}\vert_{K^{+}_{\beta,\alpha}(t)})$ denotes the topological entropy of the dynamical system $(K^{+}_{\beta,\alpha}(t), T_{\beta,\alpha}\vert_{K^{+}_{\beta,\alpha}(t)})$. Here, for a given subset $L \subseteq [0, 1]$ we set
    \begin{align*} 
        h_{\operatorname{top}}(T_{\beta,\alpha}\vert_{L}) \coloneqq \lim_{n \to \infty} \frac{\log (\lvert \tau_{\beta,\alpha}^{+}(L)\vert_{n}\rvert)}{n} 
    \end{align*}
see \cite{LM,brin_stuck_2002,Walters_1982}, and reference therein, for further details on topological entropy.  More specifically, in \Cref{sec:proof_cor_1_3_4}, we apply the following result, which is a consequence of Raith's Ledrappier-Young formula \eqref{eq:entopen} and \cite[Proposition 2.6]{KKLL}.

\begin{proposition}\label{prop:ent}
For $(\alpha, \beta) \in \Delta$ and $t \in [0,1]$,
    \begin{align}\label{eq:entropy}
        h_{\operatorname{top}}(T_{\beta,\alpha}\vert_{K_{\beta,\alpha}(t)}) = h_{\operatorname{top}}(T_{\beta,\alpha}\vert_{K^{+}_{\beta,\alpha}(t)})
    \end{align}
and hence 
    \begin{align}\label{eq:ent}
        \dim_{H}(K^{+}_{\beta,\alpha}(t)) = \frac{h_{\operatorname{top}}(T_{\beta,\alpha}\vert_{K^{+}_{\beta,\alpha}(t)})}{\log(\beta)}.
    \end{align}    
\end{proposition}

\begin{proof}
Since the set $K_{\beta,\alpha}(t)\backslash K^{+}_{\beta,\alpha}(t)$ is countable, we have $\dim_H(K_{\beta,\alpha}(t))=\dim_H(K^{+}_{\beta,\alpha}(t))$. The Ledrappier-Young formula given in \eqref{eq:ent} is therefore a direct consequence of \eqref{eq:entopen} and \eqref{eq:entropy}. The proof of \eqref{eq:entropy}, follows from a small adaptation of the proof of \cite[Proposition 2.6]{KKLL}, which we present below.

If $0 \not\in K_{\beta,\alpha}(t)$ or $t=0$, then $\mathcal{K}^{+}_{\beta,\alpha}(t) =\mathcal{K}_{\beta,\alpha}(t)$, and if $t\not \in E^{+}_{\beta,\alpha}$ then there exists a $t^*>t$ with $K^{+}_{\beta,\alpha}(t)=K^{+}_{\beta,\alpha}(t^*)$.  The first of these two statements follows directly from the definition of $K_{\beta,\alpha}(t)$ and $K_{\beta,\alpha}^{+}(t)$, and the second can be seen to hold as follows.  For every $t \not\in E^{+}_{\beta,\alpha}$ there exists a smallest natural number $N$ such that $T_{\beta,\alpha}^N(t) \in [0,t)$. Let $\delta = \min \{ p_{\beta, \alpha} - T_{\beta,\alpha}^{n}(t) \colon n \in \{ 0, 1, \dots, N-1 \} \; \text{and} \; T_{\beta,\alpha}^{n}(t) < p_{\beta, \alpha} \}$ and set $\varepsilon = \min\{t-T_{\beta,\alpha}^N(t), \delta\}/\beta^{N}$. By construction, for all $s\in (t,t+\varepsilon)$, we have that $T_{\beta,\alpha}^N(s)\in [0,t) \subset [0, s)$, and hence that $s\not\in K^{+}_{\beta,\alpha}(s)$ and $s\not\in K^{+}_{\beta,\alpha}(t) $. This implies that $K^{+}_{\beta,\alpha}(t)= K^{+}_{\beta,\alpha}(s)$ for all $s\in (t,t+\varepsilon)$.  In fact letting $t^{*} = \inf \{ s \in E^{+}_{\beta,\alpha} \colon s > t \}$, a similar justification yields that $K^{+}_{\beta,\alpha}(t)= K^{+}_{\beta,\alpha}(s)$ for all $s \in (t,t^*)$.
All-in-all, this all implies that $t^* \in E^{+}_{\beta,\alpha}$.

Therefore, it suffices to show that, if $t \in E^{+}_{\beta,\alpha} \setminus \{0\}$ and $0\in K_{\beta,\alpha}(t)$, then   $h_{\operatorname{top}}(T_{\beta,\alpha}\vert_{K_{\beta,\alpha}(t)})= h_{\operatorname{top}}(T_{\beta,\alpha}\vert_{K^{+}_{\beta,\alpha}(t)})$. To this end, suppose that $t\in E^{+}_{\beta,\alpha} \setminus \{0\}$.  In which case $\sigma^{n}(\tau^{+}_{\beta,\alpha}(t)) \neq \tau^{+}_{\beta,\alpha}(0)$, and so setting
    \begin{align*}
        \mathcal{K}_{\beta,\alpha}^{0}(t) &\coloneqq \{ \omega \in \{0,1\}^{\mathbb{N}} \colon \text{there exists} \; m \in \mathbb{N}_{0} \; \text{with} \; \sigma^{m}(\omega) = \tau^{+}_{\beta,\alpha}(0) \; \text{and} \; \tau^{+}_{\beta,\alpha}(t) \prec \sigma^{n}(\omega) \prec \tau^{-}_{\beta,\alpha}(1) \; \text{for all} \; n \in \mathbb{N}_{0} \}
    \end{align*}
and letting $\mathcal{K}_{\beta,\alpha}(t)$ denote the symbolic representation of $K_{\beta,\alpha}(t)$, namely letting
    \begin{align*}
        \mathcal{K}_{\beta,\alpha}(t) &\coloneqq \{ \omega \in \{0,1\}^{\mathbb{N}} \colon \sigma^{n}(\omega) = \tau^{+}_{\beta,\alpha}(0) \; \text{or} \; \tau^{+}_{\beta,\alpha}(t) \preceq \sigma^{n}(\omega) \prec \tau^{-}_{\beta,\alpha}(1) \; \text{for all} \; n \in \mathbb{N}_{0} \},
    \end{align*}
we have that $\mathcal{K}_{\beta,\alpha}(t)\setminus \mathcal{K}_{\beta,\alpha}^{+}(t) = \mathcal{K}_{\beta,\alpha}^{0}(t)$. Let $k \in \mathbb{N}$ be fixed, and let $\zeta \in \mathcal{K}_{\beta,\alpha}^{0}(t)\vert_{k}$ with $\zeta \neq \tau^{+}_{\beta,\alpha}(0)\vert_{k}$.  By construction, there exists $\omega \in \mathcal{K}_{\beta,\alpha}^{0}(t)$ with $\omega\vert_{k} = \zeta$.  Let $j$ be the smallest natural number such that $\sigma^{j}(\omega) = \tau^{+}_{\beta,\alpha}(0)$ and set $\nu = \omega\vert_{j-1}\xi_{j}$, where $\xi_{j}$ denotes the $j$-th letter of $\tau^{+}_{\beta,\alpha}(0)$.  Observe that 
    \begin{align}\label{eq:ent_proof}
        \tau^{+}_{\beta,\alpha}(t)\vert_{j-i} \preceq \sigma^{i}(\nu) \preceq \tau^{+}_{\beta,\alpha}(1)\vert_{j-i}
    \end{align}
for all $i \in \{0, 1, \dots, j-1\}$.  Let $i^{*} \in \{0, 1, \dots, j-1\}$ be the smallest integer such that $\sigma^{i^{*}}(\nu) = \tau^{+}_{\beta,\alpha}(t)\vert_{j-i^{*}}$, and if strict inequality holds in the lower bound of \eqref{eq:ent_proof} for all $i \in \{0, 1, \dots, j-1\}$, then set $i^{*}=j$. By the minimality of $i^{*}$, we have $\nu \, \sigma^{j-i^{*}+1}(\tau^{+}_{\beta,\alpha}(t)) = \nu\vert_{i^{*}} \tau^{+}_{\beta,\alpha}(t) \in \mathcal{K}_{\beta,\alpha}^{+}(t)$. Noting, $\nu\vert_{j-1} = \zeta\vert_{j-1}$ if $j \leq k$, and $\nu\vert_{k} = \zeta\vert_{k}$ if $j \geq k+1$, we have that
    \begin{align}\label{eq:ent_proof_2}
        \lvert \mathcal{K}_{\beta,\alpha}^{0}(t)\vert_{k} \rvert 
        \leq 1 + 2 \sum_{j = 0}^{k} \, \lvert \mathcal{K}_{\beta,\alpha}^{+}(t)\vert_{j} \rvert
        \leq 3 (k+1) \lvert \mathcal{K}_{\beta,\alpha}^{+}(t)\vert_{k} \rvert.
    \end{align}
Since $\mathcal{K}_{\beta,\alpha}(t)\setminus \mathcal{K}_{\beta,\alpha}^{+}(t) = \mathcal{K}_{\beta,\alpha}^{0}(t)$, and since by definition,
    \begin{align*}
        h_{\operatorname{top}}(T_{\beta,\alpha}\vert_{K_{\beta,\alpha}(t)})
        = \lim_{n \to \infty} \frac{\log(\mathcal{K}_{\beta,\alpha}(t)\vert_{n})}{n}
        \quad \text{and} \quad
        h_{\operatorname{top}}(T_{\beta,\alpha}\vert_{K^{+}_{\beta,\alpha}(t)})
        = \lim_{n \to \infty} \frac{\log(\mathcal{K}^{+}_{\beta,\alpha}(t)\vert_{n})}{n},
    \end{align*}
the inequality given in \eqref{eq:ent_proof_2} implies that $h_{\operatorname{top}}(T_{\beta,\alpha}\vert_{K^{+}_{\beta,\alpha}(t)}) \geq h_{\operatorname{top}}(T_{\beta,\alpha}\vert_{K_{\beta,\alpha}(t)})$. As $K^{+}_{\beta,\alpha}(t) \subseteq K_{\beta,\alpha}(t)$ and as the entropy of a subsystem cannot exceed that of its parent system, the result follows.
\end{proof}

\subsection{\texorpdfstring{$\beta$}{beta}-shifts of finite type}

In \cite{LSS16} a study of when an intermediate $\beta$-shift is of finite type was carried out. This work was continued in \cite{LSSS} where it was shown that any intermediate $\beta$-shift can be \textsl{approximated from below} by an intermediate $\beta$-shift is of finite type.  These results are summarised in the following theorem.

\begin{theorem}[{\cite{LSSS,LSS16}}]\label{thm:LSSS}
An intermediate $\beta$-shift $\Omega_{\beta,\alpha}$ is a subshift of finite type if and only if the kneading invariants $\tau_{\beta,\alpha}^{+}(p_{\beta,\alpha})$ and $\tau_{\beta,\alpha}^{-}(p_{\beta,\alpha})$ are periodic.  Moreover, given $(\beta, \alpha) \in \Delta$ and $\epsilon > 0$, there exists a $(\beta', \alpha') \in \Delta$, with $0 \leq \beta' - \beta < \epsilon$ and $\lvert \alpha - \alpha' \rvert < \epsilon$ and such that 
    \begin{enumerate}[label={\rm(\arabic*)}]
    \item $\Omega_{\beta', \alpha'}$ is a subshift of finite type,
    \item the Hausdorff distance between $\Omega_{\beta, \alpha}$ and $\Omega_{\beta', \alpha'}$ is less than $\epsilon$, and 
    \item $\Omega_{\beta, \alpha} \subseteq \Omega_{\beta', \alpha'}$.
    \end{enumerate}
\end{theorem}

\subsection{Transitivity of intermediate \texorpdfstring{$\beta$}{beta}-transformations}

An interval map $T \colon [0,1] \to [0,1]$ is said to be \textsl{transitive} if for any open subinterval $U$ of $[0,1]$ there exists an $m \in \mathbb{N}$ with $\bigcup_{k = 1}^{m} T^{k}(U) = (0,1)$. The property of transitivity will play an important part in our proof of \Cref{thm:main_2}, and thus we will utilise the following result of \cite{G1990,Palmer79} on non-transitive intermediate $\beta$-transformations.  Note the contrast in the structure of the set of $(\beta, \alpha) \in \Delta$ with $T_{\beta,\alpha}$ transitive and the set of $(\beta, \alpha) \in \Delta$ with $\Omega_{\beta, \alpha}$ of finite type, namely that the former is has positive $2$-dimensional Lebesgue measure and the latter is countable. 

\begin{theorem}[{\cite{G1990,Palmer79}}]\label{thm:G1990+Palmer79}
Let $\Delta_{\operatorname{trans}}$ denote the set of $(\beta, \alpha) \in \Delta$ with $T_{\beta,\alpha}$ transitive.  The sets $\Delta_{\operatorname{trans}}$ and $\Delta \setminus \Delta_{\operatorname{trans}}$ have positive Lebesgue measure.  Moreover, given $(\beta, \alpha) \in \Delta \setminus \Delta_{\operatorname{trans}}$, there exist
        \begin{enumerate}[label={\rm(\roman*)}]
        \item a natural number $n \geq 2$ and $k \in \{1,\ldots, n-1\}$ with $n$ and $k$ co-prime,
        \item a sequence of points $\{ b_{0}, b_{1}, \ldots, b_{2n-1}\}$ in $(0,1)$ with $b_{i} < b_{i+1}$ for all $i \in \{0, 1, \ldots, 2n-2\}$, and
        \item an $\tilde{\alpha} \in [0, 2-\beta^{n}]$,
     \end{enumerate}
such that 
        \begin{enumerate}[label={\rm(\arabic*)}]
        \item the transformation $T_{\beta^{n},\tilde{\alpha}}$ is transitive, 
        \item $T_{\beta,\alpha}^{n}(J_{i}) = J_{i}$ and $T_{\beta,\alpha}(J_{i}) = J_{i+k \bmod{n}}$, for all $i \in \{ 0, 1, \ldots, n-1 \}$, and
        \item $T_{\beta,\alpha}^{n}\vert_{J_{i}}$ is topologically conjugate to $T_{\beta^{n},\tilde{\alpha}}$, for all $i \in \{ 0, 1, \ldots, n-1 \}$, where the conjugation is linear.
     \end{enumerate}
Here, $J_{0} = [0, b_{0}] \cup [b_{2n-1}, 1]$, and $J_{i} = [b_{2i-1}, b_{2i}]$, for all $i \in \{1, 2, \ldots, n-1\}$. Further, there exists a $T_{\beta,\alpha}$-periodic point $q$ in $\mathscr{J}=\bigcup_{i = 0}^{n-2} [b_{2i}, b_{2i+1}]$, such that the orbit of $q$ under $T_{\beta,\alpha}$ is contained in $\mathscr{J}$, and for all $x$ in $\mathscr{J}$ but not in the orbit of $q$, there exists an $m \in \mathbb{N}$ such that $T_{\beta,\alpha}^{m}(x) \in \bigcup_{i = 0}^{n-1} J_{i}$.
\end{theorem}

\subsection{A sufficient condition for a dynamical set to be winning}\label{sec:HY}

To prove \Cref{thm:main_2} we not only appeal to the results of \cite{G1990,Palmer79}, but also to \cite[Theorem 2.1]{HY}, where a sufficient condition for certain dynamical sets to be winning is given. In order to state \cite[Theorem 2.1]{HY} we require the following notation.  A partition of $[0,1]$ is a collection of finitely many intervals $\{ I(i) \}_{i \in \Lambda}$, where $\Lambda = \{0, 1, \ldots, m-1\}$ for some $m \in \mathbb{N}$, with pairwise disjoint interiors such that $[0,1] = \bigcup_{i \in \Lambda} I(i)$. Here, we assume that the intervals are ordered, namely, that if $i$ and $j \in \{0, 1, \ldots, m-1\}$ with $i < j$, then, for all $x \in I(i)$ and $y \in I(j)$, we have that $x \leq y$.  Let $T \colon [0, 1] \to [0,1]$ and let $\{ I(i) \}_{i \in \Lambda}$ denote a partition of $[0, 1]$, such that $T$ restricted to $I(i)$ is monotonic and continuous for all $i \in \Lambda$. For $\xi = \xi_{1} \xi_{2} \cdots \xi_{n} \in \Lambda^{n}$, for some $n \in \mathbb{N}$, we set
    \begin{align*}
        I(\xi) = \bigcap_{i = 1}^{n} \ \{ x \in [0, 1] \colon T^{i-1}(x) \in I(\xi_{i}) \}.
    \end{align*}
If $I(\xi)$ is non-empty, we call $I(\xi)$ a \textsl{level $n$ cylinder set} of $T$, and $\xi$ an \textsl{admissible word of length $n$} with respect to the partition $\{ I(i)\}_{i \in \Lambda}$.  For $n \in \mathbb{N}_{0}$, we denote by $\Omega_{T}\vert_{n}$ the set of all admissible words of length $n$, where by convention $\Omega_{T}\vert_{0} = \{ \varepsilon \}$, and set $\Omega_{T}^{*} = \bigcup_{n \in \mathbb{N}_{0}}  \Omega_{T}\vert_{n}$.  

When $T = T_{\beta, \alpha}$ and when our partition is $\{ I(0) = [0,p_{\beta, \alpha}), I(1)=[p_{\beta, \alpha},1] \}$, for some $(\beta, \alpha) \in \Delta$, we have $\Omega_{T}\vert_{n} =\Omega_{\beta, \alpha}\vert_{n}$.  Further, for $\xi \in \{0, 1\}^{*}$, we have $I(\xi)$ is non-empty if and only if there exists an $\omega \in \Omega_{\beta, \alpha}$ with $\omega\vert_{\lvert \xi \rvert} = \xi$, and $\overline{I(\xi)} = \pi_{\beta, \alpha}( \{ \omega \in \Omega_{\beta, \alpha} \colon \omega\vert_{n} = \xi \} )$, where $\overline{I(\xi)}$ denotes the closure of $I(\xi)$.

For $\xi$ and $\nu \in \Omega_{T}^{*}$ and $c > 0$ a real number, we say that $\xi$ is \textsl{$\nu$-extendable} if the concatenation $\xi\nu$ is admissible, and say that the cylinder sets $I(\xi)$ and $I(\nu)$ are \textsl{$c$-comparable} if $c \leq \lvert I(\xi) \rvert / \lvert I(\nu) \rvert \leq 1/c$. We call $T$ is \textsl{piecewise locally $C^{1+\delta}$ expanding} if $T$ restricted to $I(i)$ is differentiable for all $i \in \Lambda$, and
    \begin{enumerate}[label={\rm(\arabic*)}]
        \item there exists a real number $\eta > 0$, such that $\lvert T'(x) \rvert > \eta$ for all $x \in I(i)$ and $i \in \Lambda$, and there exists $k \in \mathbb{N}$ and a real number $\lambda > 1$ such that $\lvert (T^{k})'(x) \rvert \geq \lambda$ for all $\xi \in \Omega_{T}\vert_{k}$ and $x \in I(\xi)$, and
        \item there exist two positive constants $\delta$ and $c$ such that for all $i \in \Lambda$ and all $x$ and $y \in I(i)$,
            \begin{align*}
                \left\lvert \frac{T'(x)}{T'(y)} - 1 \right\rvert \leq c \lvert x - y \rvert^{\delta}. 
            \end{align*}
    \end{enumerate}
We call $T$ \textsl{Markov} if for all $i$ and $j \in \Lambda$, either $T(I(i)) \cap I(j) = \emptyset$, or $I(j) \subseteq T(I(i))$. Letting $(\beta, \alpha) \in \Delta$ with $\Omega_{\beta, \alpha}$ a subshift of finite type, we set $A = \{ a_{1}, a_{2}, \ldots, a_{n} \}$ to be the set of ordered points of 
    \begin{align*}
        \{ \pi_{\beta, \alpha}(\sigma^{k}(\tau_{\beta,\alpha}^{+}(p_{\beta, \alpha}))) \colon k \in \mathbb{N}\} \cup \{ \pi_{\beta, \alpha}(\sigma^{k}(\tau_{\beta,\alpha}^{-}(p_{\beta, \alpha}))) \colon k \in \mathbb{N} \}.
    \end{align*}
The transformation $T_{\beta, \alpha}$ is a piecewise locally $C^{1+\delta}$ expanding Markov map with respect to the partition 
    \begin{align}\label{eq:Markov_Partition}
        P_{\beta,\alpha} = \{[a_{1}, a_{2}), \ldots, [a_{n-2}, a_{n-1}), [a_{n-1}, a_{n}]\}.
    \end{align}
Letting $T$ be a piecewise locally $C^{1+\delta}$ expanding map with respect to the partition $\{ I(i) \}_{i \in \Lambda}$, then for $x \in [0, 1]$, there exists an infinite word $\omega = \omega_{1} \omega_{2} \cdots  \in \Lambda^{\mathbb{N}}$, with $\omega \vert_{k} \in \Omega_{T}^{*}$ for all $k \in \mathbb{N}$ and such that $\{ x \} = \bigcap_{k \in \mathbb{N}} \overline{I(\omega\vert_{n})}$. We call $\omega$ a \textsl{symbolic representation} of $x$ with respect to the partition $\{ I(i)\}_{i \in \Lambda}$.  In the case that $T = T_{\beta,\alpha}$, for some $(\beta,\alpha) \in \Delta$, for a point $x \in [0, 1]$, symbolic representations of $x$ with respect to the partition $\{ [0,p_{\beta, \alpha}), [p_{\beta,\alpha},1] \}$ are $\tau_{\beta, \alpha}^{\pm}(x)$.  Note, all but a countable set of points have a unique symbolic representation, we denote this countable set by $E = E_{T}$.

For a fixed $x \in [0,1]$ and $\gamma \in (0, 1)$, let $\omega$ denote a symbolic representation of $x$. We denote the following geometric condition by $H_{x, \gamma}$:
    \begin{align}\label{eq:condition_1}
        \adjustlimits
        \lim_{i \to \infty} \sup_{\; u\,:\,u\;\text{and}\;u\omega\vert_{i}\in\Omega_{T}^{*}} \frac{\lvert I(u\omega\vert_{i}) \rvert}{\lvert I(u) \rvert} = 0
    \end{align}
and there exists a natural number $i^{*}$ and a real number $c>0$ such that if $i \in \mathbb{N}$ with $i \geq i^{*}$, for all $\nu$ and $\eta \in \Omega_{T}^{*}$ which are $\omega\vert_{i}$-extendable with $I(\nu)$ and $I(\eta)$ being $\gamma/4$-comparable, either
    \begin{align}\label{eq:condition_2}
        \operatorname{dist}(I(\nu\omega\vert_{i}), I(\eta\omega\vert_{i})) = 0
        \quad \text{or} \quad
        \operatorname{dist}(I(\nu\omega\vert_{i}), I(\eta\omega\vert_{i})) \geq c \operatorname{dist}(I(\nu), I(\eta)).
    \end{align}

    \begin{theorem}[{\cite{HY}}]\label{thm_HY_Thm_2.1}
        Let $T$ be a piecewise locally $C^{1+\delta}$ expanding map with respect to the partition $\{ I(i) \}_{i \in \Lambda}$, and let $x \in [0, 1]$ with symbolic representation $\omega$.
            \begin{enumerate}[label={\rm(\arabic*)}]
                \item If $H_{x, \gamma}$ is satisfied for some $\gamma \in (0, 1)$ , then the set $\{ y \in [0,1] \colon T^{k}(y) \not\in I(\omega\vert_{m}) \; \text{for all} \; k \in \mathbb{N}_{0} \} \cup E$ is $\rm ( 1/2, \gamma)$-winning for some natural number $m$.
                \item If $x \not\in E$ and if $H_{x, \gamma}$ is satisfied for any $\gamma \in (0, 1)$, then the set ${\rm BAD}_{T}(x) = \{ y \in [0, 1] \colon x \not\in \overline{\{T^k(y) \colon k \in \mathbb{N}_{0}\}}\}$ is $1/2$-winning.
            \end{enumerate}
        \end{theorem}

We conclude this section with the following proposition which we use in conjunction with \Cref{thm:G1990+Palmer79,thm_HY_Thm_2.1} and the fact that the property of winning is preserved under bijective affine transformations to prove \Cref{thm:main_2}.

    \begin{proposition}\label{prop:alpha-winning_transport}
        Let $T$ be a piecewise locally $C^{1+\delta}$ expanding interval map, let $x \in [0, 1]$ and set
            \begin{align*}
                \mathrm{BAD}(T,x) \coloneqq \{ y \in [0,1] \colon x \not\in \overline{\{T^{n}(y) \colon n \in \mathbb{N}\}}\}.
            \end{align*}
        If, for a fixed $k \in \mathbb{N}$, we have that $\mathrm{BAD}(T^k, T^m(x))$ is winning for all $m \in \{ 0, 1, \ldots, k\}$, then $\mathrm{BAD}(T,x)$ is winning.
    \end{proposition}

\begin{proof}
This follows from the fact that winning is preserved under taking countable intersections and since, by construction, $\mathrm{BAD}(T^k,x)\cap \mathrm{BAD}(T^k,T(x))\cap \cdots \cap \mathrm{BAD}(T^k,T^k(x))\subset \mathrm{BAD}(T,x)$.
\end{proof}

\section{Intermediate \texorpdfstring{$\beta$}{beta}-shifts as greedy \texorpdfstring{$\beta$}{beta}-shifts: Proof of \texorpdfstring{\Cref{thm:main}}{Theorem 1.1}}\label{sec:proof_thm_1_1}

In proving \Cref{thm:main} and \Cref{Cor_2,Cor_3}, we will investigate the question, given a fixed $\beta \in (1, 2)$, for which $\omega \in \{0,1\}^\mathbb{N}$ does there exist $(\beta', \alpha') \in \Delta$ with $\Omega^+_{\beta', \alpha'} = \{ \nu \in \{0,1\}^\mathbb{N} \colon \omega \preceq \sigma^{n}(\nu) \prec \tau_{\beta,0}^{-}(1) \; \text{for all} \; n \in \mathbb{N} \}$? Not only is this question interesting in its own right, but in classifying such words, we will be able to transfer the results from \cite{KKLL} to the setting of the intermediate $\beta$-transformations. With this in mind, we let $\mathcal{A}_\beta$ denote the set of all such words and set $\rho = \inf_{n \in \mathbb{N}_{0}} \pi_{\beta,0}(\sigma^{n}(\tau_{\beta,0}^{-}(1)))$. Further, we utilise the following notation. We let $t_{\beta, 0, c} \in (0,1)$ be such that $\dim_H(K^{+}_{\beta,0}(t))>0$ for all $t < t_{\beta, 0, c}$ and $\dim_H(K^{+}_{\beta, 0}(t)) = 0$ for all $t > t_{\beta, 0, c}$, and set $\mathcal{T}_{\beta, 0, c} = \tau_{\beta,0}^{+}(t_{\beta, c})$.

\begin{proof}[Proof of \texorpdfstring{\Cref{thm:main}}{Theorem 1.1}]
For $\beta\in(1,2)$, let $\mathcal{B}_{\beta} \coloneqq \{ \omega \in \mathcal{E}_{\beta,0}^{+} \colon \pi_{\beta,0}(\omega) \leq \rho \; \text{and} \; \omega \prec \mathcal{T}_{\beta, 0, c}\}$.  By \Cref{cor:From_greedy_to_intermediate} and the commutativity of the diagram given in \eqref{eq:commutative_diag}, observe that it is sufficient to show 
    \begin{enumerate}
        \item $\mathcal{A}_\beta\subseteq \mathcal{B}_\beta$ with equality holding for Lebesgue almost every $\beta\in(1,2)$, 
        \item there exist $\beta\in (1,2)$ such that $\mathcal{A}_\beta \neq \mathcal{B}_\beta$, and
        \item  if the quasi-greedy $\beta$-expansion of $1$ is periodic, then $\mathcal{A}_\beta = \mathcal{B}_\beta$.
    \end{enumerate}    
To show $\mathcal{A}_\beta\subseteq \mathcal{B}_\beta$, let $\beta \in (1, 2)$ and $\eta \in \mathcal{A}_\beta$ be fixed.  Setting $\omega = 1\eta$ and $\nu = 0\tau_\beta^-(1)$, we observe that they meet Conditions~(1)--(4) of \Cref{thm:BSV14} . Condition~(2) of \Cref{thm:BSV14} gives $\nu \in \Omega^{+}(\omega, \nu)$, and so, for a given $n \in \mathbb{N}_{0}$,
    \begin{align*}
        \eta \preceq \sigma^{n}(\eta) \prec 0 \tau_{\beta,0}^{-}(1) \quad \text{or} \quad 1\eta \preceq \sigma^{n}(\eta) \prec \tau_{\beta,0}^{-}(1),
    \end{align*}
yielding that $\eta \preceq \sigma^{n}(\eta) \prec \tau_{\beta,0}^{-}(1)$ for all $n \in \mathbb{N}_{0}$, namely that $\eta \in \mathcal{E}_{\beta,0}^{+}$. Condition~(2) of \Cref{thm:BSV14} also gives $\omega\in \Omega^{-}(\omega, \nu)$, and so, for a given $n \in \mathbb{N}_{0}$,
    \begin{align*}
        \eta \prec \sigma^{n}(\tau_{\beta,0}^{-}(1)) \preceq 0 \tau_{\beta,0}^{-}(1) \quad \text{or} \quad 1\eta \prec \sigma^{n}(\tau_{\beta,0}^{-}(1)) \preceq \tau_{\beta,0}^{-}(1).
    \end{align*}
This implies that $\eta \prec \sigma^{n}(\tau_\beta^-(1))$ for all $n \in \mathbb{N}_{0}$, and so $\pi_{\beta,0}(\eta)\leq \rho$. Since Condition~(3) of \Cref{thm:BSV14} holds, the topological entropy of $(\Omega(\omega,\nu), \sigma)$ is positive, and thus $\eta \prec \mathcal{T}_{\beta, 0, c}$. Therefore, $\eta \in \mathcal{B}_\beta$, and hence $\mathcal{A}_\beta\subseteq \mathcal{B}_\beta$. 

To see that $\mathcal{A}_\beta = \mathcal{B}_\beta$ for Lebesgue almost every $\beta \in (1,2)$, from the concluding remarks of \cite{Schme} we know that, for Lebesgue almost all $\beta \in (1,2)$ there is no bound on the length of blocks of consecutive zeros in the quasi-greedy $\beta$-expansion of $1$, namely $\tau_{\beta,0}^-(1)$.  This implies that $\rho = 0$, and hence that $\mathcal{B}_\beta=\{ 0^\infty \}$. Since $\tau_{\beta,0}^{\pm}(0) = 0^\infty$ it follows that $0^\infty \in \mathcal{A}_\beta$, and thus that $\mathcal{B}_\beta \subseteq \mathcal{A}_\beta$ for Lebesgue almost every $\beta \in (1,2)$. This in tandem with the fact that $\mathcal{A}_\beta\subseteq \mathcal{B}_\beta$ for all $\beta \in (1,2)$, yields that $\mathcal{A}_\beta = \mathcal{B}_\beta$ for Lebesgue almost every $\beta \in (1,2)$.

Let $\beta$ denote the algebraic number with minimal polynomial $x^5-x^4-x^3-2x^2+x+1$. An elementary calculation yields that $\tau_{\beta,0}^{-}(1)=11(100)^{\infty}$. We claim that $\xi = 00(011)^\infty \in \mathcal{B}_\beta$, but that $\xi \not\in \mathcal{A}_\beta$, namely that $\mathcal{A}_\beta \subsetneq \mathcal{B}_\beta$.  It is readily verifiable that $\xi \in \mathcal{E}_\beta^+$ and also, since $\rho=\pi_{\beta,0}((001)^\infty)$, that $\pi_{\beta,0}(\xi) < \rho$. This yields that $\{001, 011 \}^\mathbb{N} \subset \mathcal{K}^{+}_{\beta,0}(\pi_{\beta,0}(\xi))$, and hence that $h_{\operatorname{top}}(\sigma\vert_{\mathcal{K}^{+}_{\beta,0}(\pi_{\beta,0}(\xi))}) > 0$. In other words, we have $\xi \prec \mathcal{T}_{\beta,c}$, and so $\xi \in \mathcal{B}_\beta $.  By way of contradiction, suppose that $\xi \in \mathcal{A}_\beta$.  Set $\omega=0\tau_{\beta,0}^{-}(1)=011(100)^\infty$ and $\nu = 1\xi = 100(011)^\infty$. In which case $\omega$ and $\nu \in \{ \chi, \zeta\}^\mathbb{N}$ with $\chi=011$ and $\zeta=100$. Noting that $\chi$ and $\zeta$ are words of length three in the alphabet $\{0,1\}$, that $\chi\vert_{2} = 01$, $\zeta\vert_{2} = 10$, that $\chi^\infty \in \Omega^{-}(\chi^\infty, \zeta^\infty)$ and $\zeta^\infty \in \Omega^{+}(\chi^\infty, \zeta^\infty)$, but that $\omega=\chi\zeta^\infty \neq \chi^\infty$, contradicting Condition~(4) of \Cref{thm:BSV14}.

It remains to prove that if $\tau_{\beta,0}^-(1)$ is periodic, then $\mathcal{A}_\beta = \mathcal{B}_\beta$. To this end, fix $\beta \in (1, 2)$ with $\tau_{\beta,0}^{-}(1)$ periodic. Let $\xi \in \mathcal{B}_\beta $ and set $\nu=1 \xi$ and $\omega = 0\tau_\beta^-(1)$. By assumption, $\xi \in \mathcal{E}_{\beta,0}^+$, and so $\sigma(\nu) \preceq \sigma^{n}(\nu) \prec \sigma(\omega)$ and therefore $\nu \in \Omega^{+}(\omega, \nu)$, and since $\tau_{\beta,0}^{-}(1)$ is the quasi greedy $\beta$-expansion of $1$ in base $\beta$, we have $\sigma^{n}(\omega) \preceq \sigma(\omega)$ for all $n \in \mathbb{N}_{0}$. As $\pi_{\beta,0}(\xi) < \rho$ we have $ \sigma(\nu) \prec \sigma^n(\omega)$, and so $\omega \in \Omega^{-}(\omega, \nu)$. Thus, $\omega$ and $\nu$ satisfy Conditions~(1) and~(2) of \Cref{thm:BSV14}. Condition~(3) of \Cref{thm:BSV14} follows from $\xi \prec \mathcal{T}_{\beta,c}$. To conclude the proof, it suffices to show that $\omega$ and $\nu$ satisfy Condition~(4) of \Cref{thm:BSV14}.  Suppose there exist $\chi$ and $\zeta \in \{0,1\}^{*}$ of length at least three with 
    \begin{align*}
        \chi\vert_{2} = 01, \quad
        \zeta\vert_{2} = 10, \quad
        \chi^{\infty} \in \Omega^{-}(\chi^{\infty},\zeta^{\infty}),
        \quad \text{and} \quad 
        \zeta^{\infty} \in \Omega^{+}(\chi^{\infty},\zeta^{\infty}),
    \end{align*}
and such that $\omega$ and $\nu \in \{ \chi, \zeta \}^{\mathbb{N}}$. By our assumption and construction, in particular, since $\omega$ is periodic and since $\chi^{\infty} \in \Omega^{-}(\chi^{\infty},\zeta^{\infty})$, we have $\omega = \chi^\infty$. By way of contradiction, suppose that $\nu \neq \zeta^\infty$.  In which case, there exists $n \in \mathbb{N}_{0}$ such that $\sigma^{n}(\nu)\vert_{\lvert \chi \rvert + \lvert \zeta \rvert} = \chi \zeta$. Noting that  $\chi\vert_{2} = 01$ and  $\zeta\vert_{2} = 10$, this yields $\sigma(\omega) \prec \sigma^{n+1}(\nu)$ contradicting the fact that $\omega$ and $\nu$ satisfy Condition~(2) of \Cref{thm:BSV14}.
\end{proof}

\section{A Krieger embedding theorem for intermediate \texorpdfstring{$\beta$}{beta}-transformations: Proof of \texorpdfstring{\Cref{Cor_1}}{Corollary 1.2}}\label{sec:proof_cor_1_2}

To prove \Cref{Cor_1} we first show the following special case.

    \begin{theorem}\label{thm:one_perioidc}
        Let $(\beta,\alpha)\in\Delta$ be such that $\nu = \tau^{+}_{\beta, \alpha}(p_{\beta,\alpha})$ is not periodic and $\omega = \tau^{-}_{\beta, \alpha}(p_{\beta,\alpha})$ is periodic. There exists a sequence $((\beta_{n}, \alpha_{n}))_{n\in \mathbb{N}}$ in $\Delta$  with $\lim_{n\to \infty}\beta_{n} = \beta$ and $\lim_{n \to \infty}\alpha_{n} = \alpha$ and such that 
            \begin{enumerate}[label={\rm(\roman*)}]
                \item[\rm (1)] $\Omega_{\beta_{n},\alpha_n}$ is a subshift of finite type,
                \item[\rm (2)] the Hausdorff distance between $\Omega_{\beta, \alpha}$ and $\Omega_{\beta_{n}, \alpha_{n}}$ converges to zero as $n$ tends to infinity, and
                \item[\rm (3)] $\Omega_{\beta_{n},\alpha_n}\subseteq\Omega_{\beta,\alpha}$.
            \end{enumerate}
\end{theorem}

\begin{proof}
We prove this using \Cref{thm:main} and the results of \cite{KKLL}. By \Cref{thm:main}, there exist $\beta' \in (1, 2)$ and $t \in E_{\beta',0}^+$ such that $\mathcal{K}_{\beta',0}^+(t)=\Omega^+_{\beta,\alpha}$ with $\tau_{\beta',0}^{+}(t) = \sigma(\nu)$. Our goal is to find a monotonically decreasing sequence $(t_i)_{i \in \mathbb{N}}$ converging to $t$ with $t_i\in E_{\beta',0}^+$ and $t_i$ is a $T_{\beta', 0}^+$-periodic point for all $i \in \mathbb{N}$. We will first prove that $t$ is not isolated from above. For this we use the following. A finite word $s \in \{0, 1\}^{*}$ is Lyndon if $s^\infty \prec \sigma^n(s^\infty)$ for all $n \in \mathbb{N}$ with $n \neq 0 \bmod \lvert s \rvert$, and set $L_{\beta} \coloneqq \{ s\in \{0, 1\}^* \colon s \; \text{is a Lyndon word and} \; s^\infty \in \Omega_{\beta', 0} \}$.

For $s \in L_{\beta}$, let $I_{s}$ denote the half-open interval $[\pi_{\beta',0}(s0^\infty), \pi_{\beta',0}(s^\infty))$. \Cref{thm:Structure} in combination with our hypothesis that $\tau_{\beta,\alpha}^{-}(p_{\beta,\alpha}) = 0\tau_{\beta,\alpha}^{-}(1)$ is periodic, yields there exists a shortest finite word $\zeta$ with $\tau_{\beta,\alpha}^{-}(1) = \zeta^\infty$.  Letting $n$ be the length of $\zeta$, we set $\zeta^\prime$ to be the lexicographical smallest element of the set $\{ \zeta_{k} \cdots \zeta_{n} \zeta_{1} \cdots \zeta_{k-1} \colon k \in \{2, \ldots, n \}\}$, and set $y = \pi_{\beta',0}(\zeta^\prime 0^\infty)$.  By construction $\zeta^{\prime}$ is a Lyndon. Since by our hypothesis $\nu = \tau^{+}_{\beta, \alpha}(p_{\beta,\alpha})$ is not periodic and since $t < \pi_{\beta',0}({\zeta^{\prime}}^{\infty})$, we observe that $t < y$.

For $s \in L_{\beta}$, by the Lyndon property of $s$, if $x \in I_s$, then $x \not\in E_{\beta',0}^{+}$, which implies $E_{\beta',0}^{+} \cap (0,y) \subseteq (0,y) \backslash \bigcup_{s\in L_\beta} I_s$.  In fact we claim $(0,y) \backslash \bigcup_{s\in L_\beta} I_s = E_{\beta',0}^+ \cap (0,y)$. In order to prove this, let $x\in (0,y) \backslash \bigcup_{s\in L_\beta} I_s$ and suppose $x \notin E_{\beta',0}^+$.  Under this hypothesis, there exists a minimal $n \in \mathbb{N}$ such that $\sigma^{n}(\tau_{\beta', 0}^{+}(x)) \prec \tau_{\beta', 0}^{+}(x)$.  By the minimality of $n$, we have that $\xi = \tau_{\beta', 0}^{+}(x)\vert_{n}$ is a Lyndon word, and that $\tau_{\beta', 0}^{+}(x) \prec \xi^{\infty}$. If $\xi^\infty \not\in \Sigma_{\beta^\prime,0}$, then there exists $j \in \{ 1, 2, \ldots, n \}$ such that $\tau_{\beta',0}^{-}(1)\prec\sigma^j(\xi^\infty)$ where equality is excluded since $x<y$. Set $k = \tau_\beta^{-}(1) \wedge \sigma^j(\xi^\infty)$, and notice $k>n-j$; otherwise $\tau_{\beta', 0}^{+}(x)$ would not be admissible. This yields $\tau_{\beta',0}^{-}(1)=\xi_{j+1}\xi_{j+2}\cdots \xi_n (\xi_1\cdots \xi_n)^l \omega_1 \omega_2 \cdots$ with $l$ possibly $0$ but chosen so that $\omega_1 \cdots \omega_n \neq \xi_1 \cdots \xi_n$; note this is possible since $\nu$ is not periodic. Thus, $\sigma^{n-j+ln}(\tau_{\beta',0}^{-}(1)) \prec \sigma^{n-j+ln}(\sigma^{j}(\xi^\infty))=\xi^\infty$ and $\omega_1\cdots \omega_n \prec \xi$.  Hence, $\sigma^{n-j+ln}(\tau_{\beta',0}^{-}(1)) \prec \xi 0^\infty \prec \tau_{\beta',0}^{-}(x)$, contradicting the fact that we choose $x\in (0,y) \backslash \bigcup_{s\in L_\beta} I_s$. It therefore follows that $E_{\beta',0}^{+} \cap (0,y) = (0,y) \backslash \bigcup_{s\in L_\beta} I_s$ as required.

Suppose that $t$ cannot be approximated from above by elements in $E_{\beta',0}^+$, that is, there exists a real number $\epsilon > 0$ with $(t,t+\epsilon)\cap E_{\beta',0}^+=\emptyset$. Since $E_{\beta',0}^+ \cap (0,y) = (0,y) \backslash \bigcup_{s\in L_\beta} I_s$, there exists a Lyndon word $s$ with $(t,t+\epsilon) \subset I_s$, but as $I_s$ is closed from the left, $t\in I_s$, contradicting our hypothesis that $t\in E_{\beta',0}^{+}$. This implies $t$ can be approximated from above by elements in $E_{\beta',0}^{+}$, namely there exists a monotonically decreasing sequence $(t_i^\prime)_{i\in \mathbb{N}}$ of real numbers converging to $t$ with $t_i^\prime \in E_{\beta',0}^+$, for all $i \in \mathbb{N}$. If $t_i^\prime$ is not $T_{\beta',0}$-periodic for some $i \in \mathbb{}N$, then by \cite[Lemmanta~3.4 and~3.5]{KKLL}, there exists a monotonically increasing sequence of $T_{\beta',0}$-periodic points $(s_{i, j}^\prime)_{j \in \mathbb{N}}$ converging to $t_i^\prime$ with $s_{i, j} \in E_{\beta',0}^+$. For $i \in \mathbb{N}$, setting $t_i=t_i^\prime$ whenever $t_i^\prime$ is $T_{\beta',0}$-periodic, and otherwise setting $t_i=s_{i,j}^\prime$ where $s_{i,j}^\prime$ is chosen so that  $t < s_{i,j}^\prime < t_i^\prime$, the sequence $(t_i)_{i \in \mathbb{N}}$ converges to $t$ from above and $t_i$ is $T_{\beta',0}$-periodic.

Since $\omega$ is periodic with respect to the left shift map, \Cref{thm:main} implies, for each $i \in \mathbb{N}$, there exists $(\beta_{i}, \alpha_{i}) \in \Delta$ with $K_{\beta'0}^+(t_i)=\Omega^+_{\beta_{i},\alpha_i}$. Since both $\omega$ and $\tau^{+}_{\beta',0}(t_i)$ are periodic, \Cref{thm:LSSS} yields that $\Omega_{\beta_{i},\alpha_i}$ is of subshift of finite type.  Further, since $\mathcal{K}_{\beta',0}^{+}(t_i) \subseteq \mathcal{K}^{+}_{\beta',0}(t)$, it follows that $\Omega_{\beta_{i},\alpha_{i}} \subseteq \Omega_{\beta,\alpha}$ for all $i \in \mathbb{N}$.
\end{proof}

\begin{proof}[{Proof of \Cref{Cor_1}}]
Assume the setting of \Cref{Cor_1} and for ease of notation set $p = p_{\beta, \alpha}$, $\nu = \tau_{\beta, \alpha}^{+}(p)$ and $\omega = \tau_{\beta, \alpha}^{-}(p)$. By \Cref{thm:LSSS}, we have that $\Omega_{\beta, \alpha}$ is a subshift of finite type if and only if $\omega$ and $\nu$ are periodic. Since the subshift of finite type property is preserved by topological conjugation, and observing that $\Omega^{\pm}_{\beta, \alpha}$ and $\Omega^{\mp}_{\beta, 2-\beta-\alpha}$ are topologically conjugate, with conjugation map $R$, with out loss of generality we may assume that $\nu$ is not periodic. We consider the case, when $\omega$ is periodic and when $\omega$ is not periodic separately.  The former of these two cases follows from \Cref{thm:one_perioidc}, and so all that remains is to show the result for the latter case, namely when $\omega$ is not periodic. To this end, assume that $\omega$ and $\nu$ are both not periodic. Let $n \in \mathbb{N}$ be fixed, set $O_{n}^{\pm}(p) = \{ (T_{\beta, \alpha}^{\pm})^{k}(p) \colon k \in \{ 0, 1, \ldots, n-1 \} \}$, and let $\beta' \in (1, \beta)$ be such that
   \begin{align}\label{eq:def_beta_prime}
        (1-\alpha)/\beta' + (\beta+1)^{n}(\beta - \beta') < \min \{ x \in O_{n}^{+}(p) \cup O_{n}^{-}(p) \colon x > p \}.
    \end{align}
(As defined in \Cref{sec:beta-shifts}, we let $T_{\beta,\alpha}^{-} \colon x \mapsto \beta x + \alpha$ if $x \leq p$, and $x \mapsto \beta x + \alpha - 1$ otherwise, and for ease of notation, we write $T_{\beta,\alpha}^{+}$ for $T_{\beta, \alpha}$.)  Setting $p' = (1-\alpha)/\beta'$, we claim, for all $k \in \{ 1, \ldots, n-1 \}$, that either
    \begin{align}\label{eq:desired_inequalities}
        (T_{\beta', \alpha}^{\pm})^{k}(p') \leq (T_{\beta, \alpha}^{\pm})^{k}(p) \leq p \leq p'
        \quad \text{or} \quad
        (T_{\beta, \alpha}^{\pm})^{k}(p) \geq (T_{\beta', \alpha}^{\pm})^{k}(p') \geq p' \geq p.
    \end{align}
Hence, by definition and since $\omega$ and $\nu$ are not periodic, $\omega\vert_{n} = \tau_{\beta', \alpha}^{-}(p)\vert_{n}$ and $\nu\vert_{n} = \tau_{\beta', \alpha}^{+}(p)\vert_{n}$.

To prove this claim, note, for all $k \in \{ 1, \ldots, n-1 \}$, either 
    \begin{align}\label{eq:either_or}
        (T_{\beta, \alpha}^{\pm}(p))^{k} < p
        \quad \text{or} \quad 
        (T_{\beta, \alpha}^{\pm})^{k}(p) \geq  \min \{ x \in O_{n}^{+}(p) \cup O_{n}^{-}(p) \colon x > p \}.
    \end{align}
If $0 \leq y \leq x \leq p$, or if $p' \leq y \leq x \leq 1$, then 
    \begin{align}\label{eq:orbit_bound}
        0 \leq T^{\pm}_{\beta, \alpha}(x) - T^{\pm}_{\beta', \alpha}(y) = \beta x - \beta'y = \beta x - \beta y + \beta y - \beta'y \leq \beta(x-y) + (\beta-\beta').
    \end{align}
Observe that $T^{\pm}_{\beta, \alpha}(p) = T^{\pm}_{\beta', \alpha}(p')$ and
    \begin{align*}
    0 \leq (T^{\pm}_{\beta, \alpha})^{2}(p) - (T^{\pm}_{\beta', \alpha})^{2}(p')
    \leq \beta - \beta'
    \leq (\beta + 1)(\beta - \beta')
    \leq (\beta + 1)^{2}(\beta - \beta')
    \leq (\beta+1)^{n}(\beta - \beta').
    \end{align*}
Suppose, by way of induction on $m$, that 
    \begin{align*}
        0 
        \leq (T^{\pm}_{\beta, \alpha})^{m}(p) - (T^{\pm}_{\beta', \alpha})^{m}(p')
        \leq (\beta+1)^{m}(\beta-\beta')
        \leq (\beta+1)^{n}(\beta-\beta'),
    \end{align*}
for some $m \in \{2, \dots, n-2\}$. Combining \eqref{eq:def_beta_prime}, \eqref{eq:either_or} and \eqref{eq:orbit_bound} with our inductive hypothesis, we have
    \begin{align*}
        0 \leq (T^{\pm}_{\beta, \alpha})^{m+1}(p) - (T^{\pm}_{\beta', \alpha})^{m+1}(p') 
        &\leq \beta((T^{\pm}_{\beta, \alpha})^{m}{p} - (T^{\pm}_{\beta', \alpha})^{m}(p')) + (\beta-\beta')\\
        &\leq \beta(\beta+1)^{m}(\beta-\beta') + (\beta-\beta')
        \leq (\beta+1)^{m+1}(\beta-\beta')
        \leq (\beta+1)^{n}(\beta-\beta').
    \end{align*}
In other words, for all $k \in \{ 1, \ldots, n-1 \}$,
    \begin{align*}
        0 \leq (T^{\pm}_{\beta, \alpha})^{k}(p) - (T^{\pm}_{\beta', \alpha})^{k}(p') 
        \leq (\beta+1)^{k}(\beta-\beta')
        \leq (\beta+1)^{n}(\beta-\beta').
    \end{align*}
This in tandem with \eqref{eq:def_beta_prime} and \eqref{eq:either_or} proves the claim.

We observe that $(\omega, \nu) \neq (\tau_{\beta', \alpha}^{+}(p'), \tau_{\beta', \alpha}^{-}(p'))$, for if not, then since $\beta' < \beta$, this would contradict \Cref{thm:Laurent}. This implies that $\omega \neq \tau_{\beta', \alpha}^{-}(p')$ or $\nu \neq \tau_{\beta, \alpha}^{+}(p')$.  We claim that $\omega \succ \tau_{\beta', \alpha}^{-}(p')$ and $\nu \succ \tau_{\beta, \alpha}^{+}(p')$.

Consider the case when $\omega \neq \tau_{\beta', \alpha}^{-}(p')$. This implies there exists a smallest integer $m \geq n$ such that neither
   \begin{align*}
        (T_{\beta', \alpha}^{-})^{m}(p') \leq (T_{\beta, \alpha}^{-})^{m}(p) \leq p
        \quad \text{nor} \quad
        (T_{\beta, \alpha}^{-})^{m}(p) \geq (T_{\beta', \alpha}^{-})^{m}(p') \geq p'.
    \end{align*}
Using the fact that if $0 \leq y \leq x < p$ or if $p' < y \leq x \leq 1$, then $T^{-}_{\beta', \alpha}(y) \leq T^{-}_{\beta, \alpha}(x)$, in tandem with \eqref{eq:desired_inequalities}, and noting that $p<p'$, we have that
    \begin{align*}
        \tau_{\beta', \alpha}^{-}(p')\vert_{m-2}=\omega\vert_{m-2},
        \quad
        (T^{-}_{\beta',\alpha})^{m}(p')<p',
        \quad
        (T^{-}_{\beta,\alpha})^{m}(p) > p
        \quad \text{and} \quad
        (T^{-}_{\beta',\alpha})^{m}(p')\leq (T^{-}_{\beta,\alpha})^{m}(p).
    \end{align*}
Thus, $\tau_{\beta', \alpha}^{-}(p')\vert_{m-1} \prec \omega\vert_{m-1}$ and hence $\tau_{\beta', \alpha}^{-}(p') \prec \omega$. An analogous argument proves the claim when $\nu \neq \tau_{\beta, \alpha}^{+}(p')$.

Hence, we have shown, given an $n \in \mathbb{N}$, that there exists a positive $\delta \in \mathbb{R}$, such that, for all $\beta' \in (\beta-\delta, \beta)$,
    \begin{align}\label{smaller}
        \tau_{\beta', \alpha}^{\pm}(p')\vert_{n} = \tau_{\beta, \alpha}^{\pm}(p)\vert_{n}
        \quad \text{and} \quad
        \tau_{\beta', \alpha}^{\pm}(p') \prec \tau_{\beta, \alpha}^{\pm}(p),
    \end{align}
where $p'=(1-\alpha)/\beta'$.  Further, by using the fact that $\Omega^{\pm}_{\beta, \alpha}$ and $\Omega^{\mp}_{\beta, 2-\beta-\alpha}$ are topologically conjugate, with conjugating map $R$, together with \eqref{smaller}, we have that there exists a positive $\delta' \in \mathbb{R}$, such that, for all $\beta' \in (\beta-\delta', \beta)$,
    \begin{align*}
        \tau_{\beta', \alpha+\beta-\beta'}^{\pm}(p_{\beta', \alpha+\beta-\beta'})\vert_{n} = \tau_{\beta, \alpha}^{\pm}(p)\vert_{n}
        \quad \text{and} \quad
        \tau_{\beta', \alpha+\beta-\beta'}^{\pm}(p_{\beta', \alpha+\beta-\beta'}) \succ \tau_{\beta, \alpha}^{\pm}(p).
    \end{align*}
Letting $\beta' \in (\beta-\min(\delta,\delta'), \beta)$ be fixed and setting
    \begin{align*}
        q_{1} = \sup \{ a \in (\alpha, \alpha+\beta-\beta') \colon \tau_{\beta', a}^{\pm}(p_{\beta', a}) \preceq \tau_{\beta, \alpha}^{\pm}(p)\}
        \quad \text{and} \quad
        q_{2} = \inf \{ a \in (\alpha, \alpha+\beta-\beta') \colon \tau_{\beta', a}^{\pm}(p_{\beta', a}) \succeq \tau_{\beta, \alpha}^{\pm}(p)\},
    \end{align*}
by \Cref{prop:mon_cont_kneading}, we have $\alpha \leq q_{1} \leq q_{2} \leq \alpha+\beta-\beta'$ and $\tau_{\beta', a}^{\pm}(p_{\beta',a})\vert_{n} = \tau_{\beta, \alpha}^{\pm}(p)\vert_{n}$, for all $a \in [q_{1}, q_{2}]$.  Moreover, $\tau_{\beta', a}^{-}(p_{\beta',a}) \preceq \tau_{\beta, \alpha}^{-}(p) \prec \tau_{\beta, \alpha}^{+}(p) \preceq \tau_{\beta', a}^{+}(p_{\beta',a})$, for all $a \in [q_{1}, q_{2}]$, implying one of the following sets of orderings.
    \begin{align*}
        \tau_{\beta', a}^{-}(p_{\beta',a}) \prec \tau_{\beta, \alpha}^{-}(p) &\prec \tau_{\beta, \alpha}^{+}(p) \prec \tau_{\beta', a}^{+}(p_{\beta', a})\\
        \tau_{\beta', a}^{-}(p_{\beta',a}) = \tau_{\beta, \alpha}^{-}(p) &\prec \tau_{\beta, \alpha}^{+}(p) \prec \tau_{\beta', a}^{+}(p_{\beta', a})\\
        \tau_{\beta', a}^{-}(p_{\beta',a}) \prec \tau_{\beta, \alpha}^{-}(p) &\prec \tau_{\beta, \alpha}^{+}(p) = \tau_{\beta', a}^{+}(p_{\beta', a})
    \end{align*}
If either the first case occurs, the second case occurs and $\tau_{\beta', a}^{+}(p_{\beta',a})$ is not periodic, or the third  case occurs and $\tau_{\beta', a}^{-}(p_{\beta',a})$ is not periodic, then an application of \Cref{prop:mon_cont_kneading} and \Cref{thm:LSSS} yields the required result. 

This leaves two remaining sub-cases, namely when $\tau_{\beta', a}^{-}(p_{\beta',a}) = \tau_{\beta, \alpha}^{-}(p) \prec \tau_{\beta, \alpha}^{+}(p) \prec \tau_{\beta', a}^{+}(p_{\beta', a})$ with $\tau_{\beta', a}^{+}(p_{\beta', a})$ periodic, and when $\tau_{\beta', a}^{-}(p_{\beta',a}) \prec \tau_{\beta, \alpha}^{-}(p) \prec \tau_{\beta, \alpha}^{+}(p) = \tau_{\beta', a}^{+}(p_{\beta', a})$ with $\tau_{\beta', a}^{-}(p_{\beta', a})$ periodic. Let us consider the first of these two sub-cases; the second follows by an analogous arguments.

For ease of notation let $\nu' = \tau_{\beta', a}^{+}(p_{\beta',a})$ and note that by assumption $\omega = \tau_{\beta', a}^{-}(p_{\beta',a})$ and that $\omega \prec \nu \prec \nu'$.  If the map $s \mapsto \tau_{\beta', s}^{+}(p_{\beta',s})$ is continuous at $s = a$, then an application of \Cref{prop:mon_cont_kneading} and \Cref{thm:LSSS} yields the required result; if we do not have continuity at $s = a$, by \Cref{prop:mon_cont_kneading} we have that $\nu'$ is periodic with periodic $N$, for some $N \in \mathbb{N}$, and thus an application of \Cref{thm:one_perioidc} completes the proof, alternatively we may proceed as follows.

We claim that $\nu \prec \nu'\vert_{N}\omega$. Indeed if $\nu\vert_{N} \prec \nu'\vert_{N}$, the claim follows immediately, and so let us suppose that $\nu\vert_{N} = \nu'\vert_{N}$.  If $\nu_{N+1} = 0$, the claim follows, from \Cref{thm:Structure}.  On the other hand, by \Cref{thm:Structure}, if $\nu_{N+1} = 1$, then $\sigma^{N}(\nu) \succeq \nu$.  If $\sigma^{N}(\nu)\vert_{N} \succ \nu\vert_{N} = \nu'\vert_{N}$, then $\nu \succ \nu'$, contradicting our assumption that $\nu \prec \nu'$, and so  $\sigma^{N}(\nu)\vert_{N} = \nu\vert_{N}$. This implies there exists a minimal integer $m$ such that $\nu_{m N + 1} = 0$ and $\nu\vert_{mN} = \nu'\vert_{mN}$; otherwise $\nu$ would be periodic. However, this together with \Cref{thm:Structure}, yields that $\nu \preceq \sigma^{(m-1)N}(\nu) = \nu\vert_{N}\sigma^{mN}(\nu) \preceq \nu\vert_{N}\omega = \nu'\vert_{N}\omega$, as required. 

To complete the proof of this sub-case we appeal once more to \Cref{prop:mon_cont_kneading} which together with the above implies that there exists a real number $\delta > 0$ such that for all $a' \in (a-\delta,a)$ we have $\tau_{\beta',a'}^{-}(p_{\beta',a'}) \prec \omega$ and $\nu \prec \tau_{\beta',a'}^{+}(p_{\beta',a'}) \prec \nu'\vert_{N}\omega \prec \nu'$. An application of \Cref{thm:LSSS} yields the required result.
\end{proof}

In the above proof, it is critical that $\omega$ and $\nu$ are not periodic, as this allows us to construct $\beta'$, $q_{1}$ and $q_{2}$ so that $\omega$ and $\nu$ are sufficiently close to $\tau_{\beta',a'}^{-}(p_{\beta',a'})$ and $\tau_{\beta',a'}^{+}(p_{\beta',a'})$, respectively, for all $a' \in [q_{1}, q_{2}]$. However, under the assumption that $\nu$ is periodic we may not use our construction to build such $\beta'$ and hence $q_{1}$ and $q_{2}$.  Indeed, the strict inequalities in \Cref{eq:either_or} no longer hold, and thus the ordering given in \eqref{smaller} fails.

\section{Survivor sets of intermediate \texorpdfstring{$\beta$}{beta}-transformations: Proof of \texorpdfstring{\Cref{Cor_2,Cor_3}}{Corollaries 1.3 and 1.4}}\label{sec:proof_cor_1_3_4}

Here, we examine open dynamical systems on the unit interval with a hole at zero and where the dynamics is driven by an intermediate \mbox{$\beta$-transformation}. With the aid of \Cref{thm:main} we can relate such open dynamical system to open dynamical systems driven by greedy \mbox{$\beta$-transformations}. This allows us to transfer the results of \cite{KKLL} and \cite{AK} on isolated points in $E_{\beta,\alpha}^+$, the Hausdorff dimension of survivor sets, and the critical point of the dimension function from the Greedy case to the intermediate case. For readability, we omit the $0$ in notation of $\pi_{\beta,0}$, $E_{\beta}^+=E_{\beta,0}^+$, and so on, and thus write $\pi_{\beta}$ for $\pi_{\beta,0}$, $E_{\beta}^+$ for $E_{\beta,0}^+$, and so forth.

By \Cref{thm:Parry_converse,thm:Structure}, given $(\beta,\alpha) \in \Delta$, there exists a unique $\beta^\prime \in (1, 2)$ with $\tau_{\beta,\alpha}^-(1)=\tau_{\beta^\prime}^-(1)$. Thus, we define a function $u \colon \Delta \to (1, 2)$ by $u(\beta,\alpha) \coloneqq \beta^\prime$, and let $\tilde{\pi}_{\beta,\alpha} \coloneqq \pi_{u(\beta,\alpha)} \circ \tau_{\beta,\alpha}^{+}$.  Correlations of the systems $(T_{\beta,\alpha},[0,1])$ and $(T_{u(\beta,\alpha)},K_{u(\beta,\alpha)}^+(\tilde{\pi}_{\beta,\alpha}(0)))$ are expressed in the following proposition.

\begin{proposition}\label{prop:char}
Let $(\beta,\alpha)\in\Delta$ and let $\beta^\prime=u(\beta,\alpha)$.
    \begin{enumerate}
        \item  $\tilde{\pi}_{\beta,\alpha}([0,1])=K_{\beta^\prime}^+(\tilde{\pi}_{\beta,\alpha}(0))$ and $\tilde{\pi}_{\beta,\alpha}(E_{\beta,\alpha}^+)=E_{\beta^\prime}^+\cap[\tilde{\pi}_{\beta,\alpha}(0),1]$.
        \item For every $x\in E_{\beta,\alpha}^+$ we have that $x$ is isolated in $E_{\beta,\alpha}^+$ if and only if $\tilde{\pi}_{\beta,\alpha}(x)$ is isolated in $E_{\beta^\prime}^+$.
        \item For $t\in(0,1)$, we have $ \tilde{\pi}_{\beta,\alpha}(K_{\beta,\alpha}^+(t))=K_{\beta^\prime}^+(\tilde{\pi}_{\beta,\alpha}(t)) $.
        \item For $t\in(0,1)$, we have $\dim_H(K_{\beta,\alpha}^+(t))=(\log(\beta^\prime)/\log(\beta)) \dim_H(K_{\beta^\prime}^+(\tilde{\pi}_{\beta,\alpha}(t)))$.
    \end{enumerate}
\end{proposition}

\newpage

\begin{proof}
Let us begin by proving Part~(1). To this end, observe that $\tilde{\pi}_{\beta,\alpha}$ is monotonic, since it is a composition of monotonic functions, and so, $\tilde{\pi}_{\beta,\alpha}(T_{\beta,\alpha}^n(x)) \geq \tilde{\pi}_{\beta,\alpha}(0)$ for all $x \in [0,1]$ and $n \in \mathbb{N}_{0}$. By the fact that the diagrams in \eqref{eq:commutative_diag} commute, we have for all $x \in [0,1]$ and $n \in \mathbb{N}$, that 
    \begin{align*}
        T_{\beta'}^{n}(\tilde{\pi}_{\beta,\alpha}(x))
        = T_{\beta'}^{n}(\pi_{\beta'}(\tau^{+}_{\alpha,\beta}(x)))
        = \pi_{\beta'}(\sigma^{n}(\tau^{+}_{\alpha,\beta}(x)))
        = \pi_{\beta'}(\tau^{+}_{\alpha,\beta}(T^{n}_{\beta,\alpha}(x))
        = \tilde{\pi}_{\beta,\alpha}(T^{n}_{\beta,\alpha}(x)).
    \end{align*}
Combining the above, we may conclude that $\tilde{\pi}_{\beta,\alpha}([0,1]) \subseteq  K_{\beta^\prime}(\tilde{\pi}_{\beta,\alpha}(0))$. To prove that equality holds, namely that $\tilde{\pi}_{\beta,\alpha}([0,1]) =  K_{\beta^\prime}(\tilde{\pi}_{\beta,\alpha}(0))$, using the commutativity of the diagrams in \eqref{eq:commutative_diag}, we observe that $x \in K_{\beta^\prime}(\tilde{\pi}_{\beta,\alpha}(0))$, if and only if, $\pi_{\beta'}(\tau^{+}_{\beta,\alpha}(0)) \leq \pi_{\beta'}(\sigma^{n}(\tau^{+}_{\beta'}(x)) \leq \pi_{\beta'}(\tau_{\beta'}^{-}(1))$ for all $n \in \mathbb{N}_{0}$. Since $\pi_{\beta'}$ is injective on $\Omega_{\beta'}$ and monotonic on $\{0,1\}^{\mathbb{N}}$, and since $\tau_{\beta,\alpha}^-(1)=\tau_{\beta^\prime}^-(1)$, it follows that $\tau_{\beta'}(x) \in \Omega^{+}_{\beta, \alpha}$.  In other words, there exists a $y \in [0,1]$ such that $\tilde{\pi}_{\beta,\alpha}(y) = \pi_{\beta'}(\tau^{+}_{\beta,\alpha}(y)) = x$, yielding the first statement of Part~(1). Let us now prove the second statement. If $x \in \tilde{\pi}_{\beta,\alpha}(E_{\beta,\alpha}^{+})$, then there exists a $y \in E_{\beta,\alpha}^{+} \subset [0,1]$ with $\tilde{\pi}_{\beta,\alpha}(y) = x$. This in tandem with the fact that the the diagrams in \eqref{eq:commutative_diag} are commutative, and since the maps $\tau_{\beta,\alpha}^{+}$ and $\pi_{\beta'}$ are monotonic, we have 
    \begin{align*}
        T_{\beta'}^{n}(x)
        = T_{\beta'}^{n}( \pi_{\beta'}(\tau_{\beta,\alpha}^{+}(y)))
        = \pi_{\beta'}(\sigma^{n}(\tau_{\beta,\alpha}^{+}(y)))
        = \pi_{\beta'}(\tau_{\beta,\alpha}^{+}(T_{\beta,\alpha}^{n}(y)))
        \geq \pi_{\beta'}(\tau_{\beta,\alpha}^{+}(y))
        = x.
    \end{align*}
Since $y \in [0,1]$ and $\tilde{\pi}_{\beta,\alpha}(y) = x$, and since $\tilde{\pi}_{\beta,\alpha}$ is monotonic, $x \geq \tilde{\pi}_{\beta,\alpha}(0)$.  This, together with the fact that $E_{\beta'}^{+} \subseteq [0, 1]$, yields  $\tilde{\pi}_{\beta,\alpha}(E_{\beta,\alpha}^+) \subseteq E_{\beta^\prime}^+\cap[\tilde{\pi}_{\beta,\alpha}(0),\tilde{\pi}_{\beta,\alpha}(1)]$.  To see that $\tilde{\pi}_{\beta,\alpha}(E_{\beta,\alpha}^+) \supseteq E_{\beta^\prime}^+\cap[\tilde{\pi}_{\beta,\alpha}(0),1]$ let $x \in E_{\beta'}^{+}$ with $x \geq \tilde{\pi}_{\beta,\alpha}(0)$.  By definition $E_{\beta'}^{+}$ and the commutativity of the diagrams in \eqref{eq:commutative_diag}, 
    \begin{align*}
        \pi_{\beta,\alpha}(\tau_{\beta,\alpha}^{+}(0))
        \leq \pi_{\beta,\alpha}(\tau_{\beta'}^{+}(x))
        \leq T^{n}_{\beta,\alpha}(\pi_{\beta,\alpha}(\tau_{\beta'}^{+}(x)))
        \leq \pi_{\beta,\alpha} (\tau_{\beta'}^{-}(1))
        \leq \pi_{\beta,\alpha} (\tau_{\beta,\alpha}^{-}(1)).
    \end{align*}
In other words $\pi_{\beta,\alpha}(\tau_{\beta'}^{+}(x)) \in E_{\beta,\alpha}^+$. Since $\tilde{\pi}_{\beta,\alpha}$ is invertible on $[0,1)$ with inverse $\pi_{\beta,\alpha} \circ \tau^{+}_{\beta'}$ the result follows.

Part~(2) follows from Part~(1) using the fact that $\tilde{\pi}_{\beta,\alpha}$ is monotonic and injective on $[0,1)$. Part~(3) follows using analogous argument to those used above to proof Part~(1), and Part~(4) follows from \Cref{prop:ent} and Part~(3) in the following way.
    \[
        \dim_H(K_{\beta,\alpha}^+(t))=\frac{h_{\operatorname{top}}(T_{\beta,\alpha}\vert_{K^+_{\beta,\alpha}(t)})}{\log(\beta)}=\frac{h_{\operatorname{top}}(T_{\beta^\prime}\vert_{K^+_{\beta^\prime}(\tilde{\pi}_{\beta,\alpha}(t))})}{\log(\beta)}=\frac{\log(\beta^\prime)}{\log(\beta)}\dim_H(K_{\beta^\prime}^+(\tilde{\pi}_{\beta,\alpha}(t))).
        \qedhere
    \]
\end{proof}

For the next proposition we will require the following analogue of the map $\tilde{\pi}_{\beta,\alpha}$, namely $\tilde{\pi}_{\beta,\alpha}^{-} \coloneqq \pi_{u(\beta,\alpha)} \circ \tau_{\beta,\alpha}^{-}$.  Note in the previous proposition, we could have also used the map $\tilde{\pi}_{\beta,\alpha}^{-}$ instead of $\tilde{\pi}_{\beta,\alpha}$ since  they coincide on all points considered in (1)--(4).  However, in the proof of \Cref{prop:char} we would need to replace $T_{\beta,\alpha}$ by $T_{\beta,\alpha}^{-}$, $T_{\beta'}$ by $T_{\beta'}^{-}$, $\tau_{\beta,\alpha}^{\pm}$ by $\tau_{\beta,\alpha}^{\mp}$ and $\tau_{\beta'}^{\pm}$ by $\tau_{\beta'}^{\mp}$, making it notionally heavy, and thus for ease of notation we use $\tilde{\pi}_{\beta,\alpha}$.

\begin{proposition}\label{prop:char(5)}
For all $(\beta,\alpha)\in\Delta$, we have that $\tilde{\pi}^{-}_{\beta,\alpha}(t_{\beta,\alpha,c})=t_{u(\beta,\alpha),c}$.
\end{proposition}

\begin{proof}
Observe that there exists a sequence of real numbers $(t_{n})_{n\in \mathbb{N}}$ with $t_n \in E_{\beta,\alpha}^{+}$ such that $t_n < t_{\beta,\alpha,c}$ and $\lim_{n\to \infty} t_{n} = t_{\beta,\alpha,c}$; otherwise the dimension function would be constant around  $t_{\beta,\alpha,c}$ contradicting its definition. Define $\hat{t}_n =\tilde{\pi}^{-}_{\beta,\alpha}(t_n)$. By Proposition \ref{prop:char} Part~(3), for all $n\in \mathbb{N}$, we have $\tilde{\pi}^{-}_{\beta,\alpha}(K_{\beta,\alpha}^+(t_n))=K_{u(\beta,\alpha)}^+(\hat{t}_n)$. An application of Proposition \ref{prop:char} Part~(4) together with our remarks directly preceding this Proposition, yields for $n \in \mathbb{N}$,
    \begin{align*}
        \dim_H(K_{u(\beta,\alpha)}^+(\hat{t}_n))>0, \quad
        \dim_H(K_{\beta,\alpha}^+(t_{\beta,\alpha,c}))=0,
        \quad \text{and} \quad
        \dim_H(K_{u(\beta,\alpha)}^+(\tilde{\pi}^{-}_{\beta,\alpha}(t_{\beta,\alpha,c}))=0.
    \end{align*}
As $\pi_{u(\beta,\alpha)}$ is continuous and $\tau_{\beta,\alpha}^{-}$ is left continuous, $\tilde{\pi}^{-}_{\beta,\alpha}$ is left continuous, and so $\tilde{\pi}^{-}_{\beta,\alpha}(\lim_{n\to \infty}(t_n))= \lim_{n\to \infty} \tilde{\pi}^{-}_{\beta,\alpha}(t_N)$.  This implies that $\tilde{\pi}^{-}_{\beta,\alpha}(t_{\beta,\alpha,c}) = \lim_{n\to \infty} \hat{t}_N$, and hence that $\tilde{\pi}^{-}_{\beta,\alpha}(t_{\beta,\alpha,c})=t_{u(\beta,\alpha),c}$.
\end{proof}

The value of $t_{u(\beta,\alpha),c}$ is explicitly given in \cite{KKLL} when $\tau_{\beta,\alpha}^-(1)$ is balanced. For all other cases see \cite{AK}.

A word $\omega = \omega_{1}\omega_{2} \cdots \in \{0,1\}^{\mathbb{N}}$ is called \textsl{balanced} if $\lvert(\omega_{n} + \omega_{n+1} + \cdots + \omega_{n+m}) - (\omega_{k+n} + \omega_{k+n+1} + \cdots + \omega_{k+n+m}) \rvert \leq 1$ for all $k$, $n$ and $m \in \mathbb{N}_{0}$ with $n \geq 1$. Following notation of \cite{KKLL} and \cite{Schme}, we let 
    \begin{align*}
        C_3 \coloneqq\{ \beta\in(1,2) \colon \text{the length of consecutive zeros in} \; \tau_\beta^-(1) \; \text{is bounded} \}
        \;\; \text{and} \;\;
        C \coloneqq \{ \beta \in (1,2) \colon \tau_\beta^-(1) \; \text{is balanced} \}.
    \end{align*}
For every $(\beta,\alpha) \in \Delta$ with $\alpha>0$, we have that $u(\beta,\alpha)\in C_3$. By \cite[Theorem 3.12]{KKLL}, for $\beta \in C_3$, there exists $\delta>0$ such that $E_{\beta}^+\cap[0,\delta]$ contains no isolated points. With this in mind and, for $\beta \in (1,2)$, setting $\delta(\beta) \coloneqq \sup \{ \delta \in [0, 1] \colon  E_{\beta}^+\cap[0,\delta] \; \text{contains no isolated points} \}$, we have the following corollary of Proposition \ref{prop:char}.

\begin{corollary}\label{cor:isoloated_pts}
Let $(\beta,\alpha)\in \Delta$ with $\alpha>0$.  If $\tilde{\pi}_{\beta,\alpha}(0)<\delta(u(\beta,\alpha))$ then there exists a $\delta>0$ such that $E_{\beta,\alpha}^+\cap[0,\delta]$ contains no isolated points. Further, if $u(\beta,\alpha) \in C$, then $\delta(u(\beta,\alpha))=1$ and $E_{\beta,\alpha}^+$ contains no isolated points.
\end{corollary}

\begin{proof}
The first statement follows from \Cref{prop:char} Parts (1) and (2), and \cite[Theorem 3.12]{KKLL}.  The second statement follows from \cite[Theorem 3]{KKLL}, which states that if $\beta\in C$ then $E_\beta$ does not contain any isolated points.
\end{proof}

\begin{proof}[Proof of \texorpdfstring{\Cref{Cor_2}}{Corollary 1.3}]
In \cite{MR0166332} an absolutely continuous invariant measure of $T_{\beta,\alpha}$ is constructed, and in \cite{Hof} it is shown that this measure is ergodic. (In fact it is shown that it is maximal, and the only measure of maximal entropy.)  This yields, given an $m \in \mathbb{N}$, that for almost all $x \in [0,1]$, there exists $n_{x} = n \in \mathbb{N}_{0}$ such that $T_{\beta,\alpha}^{n}(x) \in [0, 1/m)$, and hence that $K_{\beta,\alpha}^{+}(1/m)$ is a Lebesgue null. Since $E_{\beta,\alpha}^{+} \setminus \{0\} \subseteq \cup_{m=1}^{\infty} K_{\beta,\alpha}^{+}(1/m)$, by subadditivity of the Lebesgue measure, it follows that $E_{\beta,\alpha}^{+}$ is a Lebesgue null set. The statement on the isolated points of $E_{\beta,\alpha}^{+}$ follows from \Cref{cor:isoloated_pts}.
\end{proof}

\begin{proof}[Proof of \texorpdfstring{\Cref{Cor_3}}{Corollary 1.4}]
This is a direct consequence of \Cref{prop:char} Part~(4) and \cite[Theorem~A (ii)]{KKLL}.
\end{proof}

\begin{proof}[Proof of \texorpdfstring{\Cref{Cor_E_beta_alpha}}{Corollary 1.5}]
Let $(\beta, \alpha) \in \Delta$ with $\Omega_{\beta,\alpha}$ a subshift of finite type and $T_{\beta,\alpha}$ transitive, let $P_{\beta,\alpha}$ denote the Markov partition of $T_{\beta,\alpha}$ defined in \eqref{eq:Markov_Partition}, and for $n\in \mathbb{N}$, let $\Omega_{T_{\beta,\alpha}}\vert_{n}$ denote the set of all length $n$ admissible words of $T_{\beta,\alpha}$ with respect to the partition $P_{\beta,\alpha}$.  Fix $n \in \mathbb{N}$ sufficient large, set $\omega$ to be lexicographically the smallest word in $\Omega_{T_{\beta,\alpha}}\vert_{n}$, let $a_{n} = a_{\beta, \alpha, n} = \sup I(\omega)$ and let $\nu \in \Omega_{T_{\beta,\alpha}}\vert_{n}$ with $\nu \succ \omega$. By transitivity and the Markov property, there exist $k \in \mathbb{N}$ and $\xi \in \Omega_{\beta,\alpha}\vert_{k}$ with $k > n$, $I(\xi) \subset I(\omega)$, $T_{\beta,\alpha}^{k-n}(I(\xi)) = I(\nu)$, and $T_{\beta,\alpha}^j(I(\xi))$ an interval and $T_{\beta,\alpha}^{j}(x) \geq a_{n}$ for all $j \in \{1,2, \dots, k - n\}$ and $x \in I(\xi)$.  In other words, there exists a linear scaled copy of $K_{\beta,\alpha}^{+}(a_{n}) \cap I(\nu)$ in $I(\xi) \cap E_{\beta,\alpha}^{+}$.   Namely, we have
    \begin{align}\label{eq:scaling_func}
        f_{\beta,\alpha,\nu,n}(K_{\beta,\alpha}^{+}(a_{n}) \cap I(\nu)) \subseteq I(\xi) \cap E_{\beta,\alpha}^{+},
    \end{align}
where $f_{\beta, \alpha,\nu, n} = f_{\beta,\alpha,\chi(\xi_{1})} \circ f_{\beta,\alpha,\chi(\xi_{2})} \circ \cdots \circ f_{\beta,\alpha,\chi(\xi_{k-n})}$ and $\chi \colon \Omega_{T_{\beta,\alpha}}\vert_{1} \to \{ 0, 1\}$ is defined by
    \begin{align*}
        \chi(a) = \begin{cases}
            0 & \text{if} \; I(a) \subseteq [0, p_{\beta,\alpha}],\\
            1 & \text{otherwise.}
        \end{cases}
    \end{align*}
Here, we recall that $f_{\beta,\alpha,0}(x) = \beta^{-1}x-\alpha\beta^{-1}$ and $f_{\beta,\alpha, 1}(x) = \beta^{-1}x-(\alpha-1)\beta^{-1}$ for $x \in [0,1]$. With the above at hand, we may conclude that
    \begin{align*}
        \dim_H (E_{\beta,\alpha}^{+})
        \geq \max \{ \dim_H(K_{\beta,\alpha}^{+}(a_{n}) \cap I(\nu)) \colon \nu \in \Omega_{T_{\beta,\alpha}}\vert_{n} \}
        = \dim_H(K_{\beta,\alpha}^{+}(a_{n})).
    \end{align*}
Since $n$ was chosen sufficiently large but arbitrarily, this in tandem with \Cref{Cor_3} implies $\dim_H (E_{\beta,\alpha}^{+}) = 1$, since $a_n$ converges to zero as $n$ tends to infinity.  

Since Hausdorff dimension is preserved under taking linear transformations, an application of \Cref{thm:G1990+Palmer79,thm:LSSS}, yields for $(\beta, \alpha) \in \Delta$ with $\Omega_{\beta,\alpha}$ a subshift of finite type, that $\dim_H (E_{\beta,\alpha}^{+}) = 1$. 

To conclude, let $(\beta,\alpha) \in \Delta$ be chosen arbitrarily, and let $((\beta_{n},\alpha_{n}))_{n \in \mathbb{N}}$ denote the sequence of tuples given in \Cref{Cor_1} converging to $(\beta, \alpha)$. Set $\tilde{\pi}_{\beta, \alpha}^{(n)} = \pi_{\beta,\alpha} \circ \tau^{+}_{\beta_{n},\alpha_{n}}$, and for $t$ and $s \in [0,1]$ with $t < s$, let
    \begin{align*}
        K_{\beta,\alpha}^{+}(t, s) \coloneqq \{ x \in[0, 1) \colon T_{\beta,\alpha}^{n}(x) \not \in [0,t) \cup (s, 1] \; \textup{for all} \; n \in \mathbb{N}_{0} \}
    \end{align*}
By \Cref{Cor_1}, \Cref{thm:Structure}, and the commutativity of the diagram in \eqref{eq:commutative_diag}, we may choose $((\beta_{n},\alpha_{n}))_{n \in \mathbb{N}}$ so that $(\pi_{\beta,\alpha}^{(n)}(0))_{n \in \mathbb{N}}$ is a monotonically decreasing sequence converging to zero, and $(\pi_{\beta,\alpha}^{(n)}(1))_{n \in \mathbb{N}}$ is a monotonically increasing sequence converging to one. Thus, by construction, for $n$ and $l \in \mathbb{N}$ with $l \geq n$,
    \begin{align*}
        K_{\beta,\alpha}^{+}(\pi_{\beta,\alpha}^{(n)}(0), \pi_{\beta,\alpha}^{(n)}(1)) \subseteq K_{\beta,\alpha}^{+}(\pi_{\beta,\alpha}^{(n)}(0), \pi_{\beta,\alpha}^{(l)}(1)),
        \quad \text{and} \quad 
        K_{\beta,\alpha}^{+}(\pi_{\beta,\alpha}^{(n)}(0)) = \bigcup_{m \in \mathbb{N}} K_{\beta,\alpha}^{+}(\pi_{\beta,\alpha}^{(n)}(0), \pi_{\beta,\alpha}^{(m)}(1)).
    \end{align*}
Hence, by countable stability of the Hausdorff dimension and \Cref{Cor_3},
    \begin{align}\label{eq:limit_hausdorff_proj_0,1}
        \lim_{n \to \infty} \dim_{H}(K_{\beta,\alpha}^{+}(\pi_{\beta,\alpha}^{(n)}(0),\pi_{\beta,\alpha}^{(n)}(1))) = 1.   
    \end{align}
Via analogous arguments to those given in the proof of \Cref{prop:char}, we have the following.
    \begin{enumerate}
        \item[($1^{*}$)]  $\tilde{\pi}_{\beta,\alpha}^{(n)}([0,1])=K_{\beta,\alpha}^+(\tilde{\pi}_{\beta,\alpha}^{(n)}(0), \tilde{\pi}_{\beta,\alpha}^{(n)}(1))$ and $\tilde{\pi}_{\beta,\alpha}^{(n)}(E_{\beta_{n},\alpha_{n}}^+)=E_{\beta, \alpha}^+\cap[\tilde{\pi}_{\beta,\alpha}^{(n)}(0),\tilde{\pi}_{\beta,\alpha}^{(n)}(1)]$.
        \item[($3^{*}$)] For $t\in(0,1)$, we have $ \tilde{\pi}_{\beta,\alpha}^{(n)}(K_{\beta_{n},\alpha_{n}}^+(t))=K_{\beta,\alpha}^+(\tilde{\pi}_{\beta,\alpha}^{(n)}(t), \tilde{\pi}_{\beta,\alpha}^{(n)}(1)) $.
    \end{enumerate}
For $k \in \mathbb{N}$ sufficiently large and $\nu \in \Omega_{T_{\beta_{n},\alpha_{n}}}\vert_{k}$, setting $a_{n,k} = a_{\beta_{n},\alpha_{n},k}$, from the equalities given in \eqref{eq:alt_IFS} and \eqref{eq:scaling_func}, the commutativity of the diagram in \eqref{eq:commutative_diag} and ($1^{*}$), we have that
    \begin{align*}
        f_{\beta,\alpha, \nu, k}(
        \pi_{\beta,\alpha}^{(n)}( K_{\beta_{n},\alpha_{n}}^{+}(a_{n,k}) \cap I(\nu))) 
        =\pi_{\beta,\alpha}^{(n)}( f_{\beta_{n},\alpha_{n}, \nu, k}( K_{\beta_{n},\alpha_{n}}^{+}(a_{n,k}) \cap I(\nu)))
        \subseteq \pi_{\beta,\alpha}^{(n)}(E_{\beta_{n},\alpha_{n}}^{+})
        \subseteq E_{\beta,\alpha}.
    \end{align*}
This in tandem with ($3^{*}$), the fact that there exists $\nu \in \Omega_{T_{\beta_{n},\alpha_{n}}}\vert_{k}$ with 
    \begin{align*}
        \dim_{H}(\pi_{\beta,\alpha}^{(n)}( K_{\beta_{n},\alpha_{n}}^{+}(a_{n,k}) \cap I(\nu))) = \dim_{H}(\pi_{\beta,\alpha}^{(n)}( K_{\beta_{n},\alpha_{n}}^{+}(a_{n, k}))),
    \end{align*}
and that Hausdorff dimension is invariant under linear scaling, implies that
    \begin{align*}
        \dim_{H}(K_{\beta,\alpha}^{+}(\pi_{\beta,\alpha}^{(n)}(a_{n,k}),
        \pi_{\beta,\alpha}^{(n)}(1)))
        =\dim_{H}(\pi_{\beta,\alpha}^{(n)}( K_{\beta_{n},\alpha_{n}}^{+}(a_{n,k}))) \leq \dim_{H}(E_{\beta,\alpha}).
    \end{align*}
This in tandem with \Cref{Cor_3}, the equality given in \eqref{eq:limit_hausdorff_proj_0,1}, the observations that, for $n \in \mathbb{N}$, the sequence $(\pi_{\beta,\alpha}^{(n)}(a_{n,k}))_{k \in \mathbb{N}}$ is monotonically decreasing with $\lim_{k \to \infty} \pi_{\beta,\alpha}^{(n)}(a_{n,k}) =\pi_{\beta,\alpha}^{(n)}(0)$, and for $l$ and $m \in \mathbb{N}$ with $l \geq m$, 
    \begin{align*}
        K_{\beta,\alpha}^{+}(\pi_{\beta,\alpha}^{(n)}(a_{n,m}),
        \pi_{\beta,\alpha}^{(n)}(1)) \subseteq K_{\beta,\alpha}^{+}(\pi_{\beta,\alpha}^{(n)}(a_{n,l}), \pi_{\beta,\alpha}^{(n)}(1)),
        \quad \text{and} \quad
        K_{\beta,\alpha}^{+}(\pi_{\beta,\alpha}^{(n)}(0),
        \pi_{\beta,\alpha}^{(n)}(1)) = \bigcup_{k \in \mathbb{N}} K_{\beta,\alpha}^{+}(\pi_{\beta,\alpha}^{(n)}(a_{n,k}),
        \pi_{\beta,\alpha}^{(n)}(1)),
    \end{align*}
and the countable stability of the Hausdorff dimension, yields the required result.
\end{proof}

\subsection*{Examples and applications}

\begin{figure}[t]
  \centering
  \begin{subfigure}[b]{0.32\textwidth}
    \includegraphics[width=\textwidth]{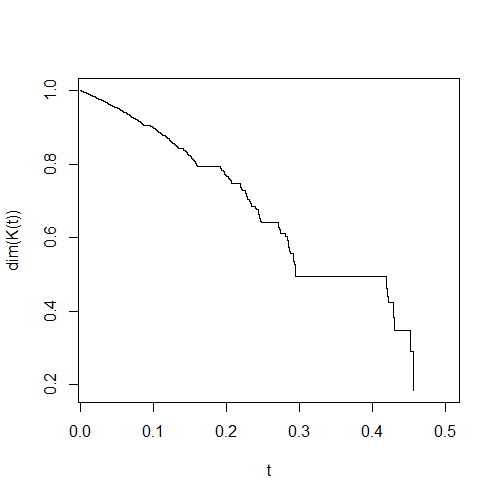}
  \end{subfigure}
  \hfill
  \begin{subfigure}[b]{0.32\textwidth}
    \includegraphics[width=\textwidth]{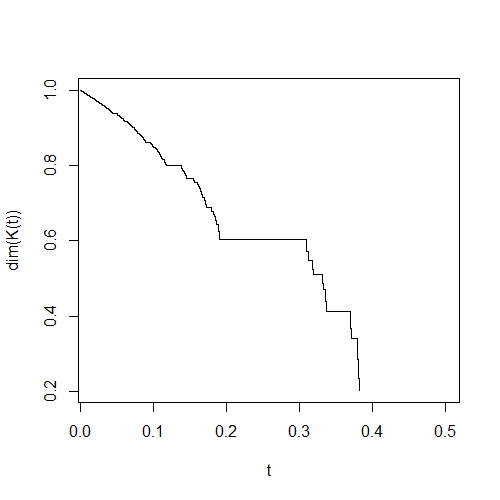}
  \end{subfigure}
    \hfill
  \begin{subfigure}[b]{0.32\textwidth}
    \includegraphics[width=\textwidth]{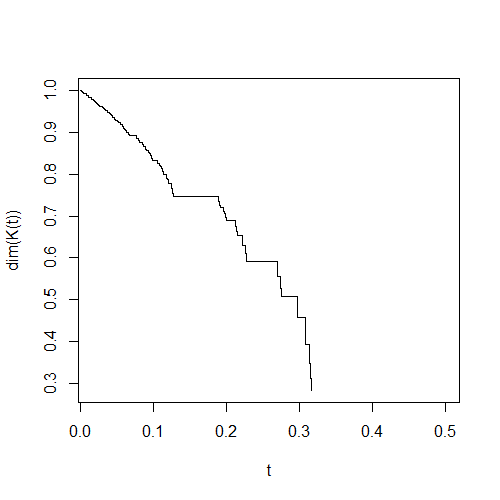}
  \end{subfigure}
  \caption{Graphs of $\eta_{\beta,\alpha}$: on the left, $\eta_{\beta}(t)$ with $\beta$ such that $\tau_\beta^-(1)=(110)^\infty$; in the middle, $\eta_{\beta,\alpha}$ for $(\beta, \alpha) \in \Delta$ with $\tau_{\beta,\alpha}^-(1)=(110)^\infty$ and $\tau_{\beta,\alpha}^+(0)=(001)^\infty$; on the right, $\eta_{\beta}$ for $(\beta, \alpha) \in \Delta$ such that $\tau_{\beta,\alpha}^-(1)=(10)^\infty$ and $\tau_{\beta,\alpha}^+(0)=(0001)^\infty $.}\label{fig:dim}
\end{figure}

Let $(\beta,\alpha) \in \Delta$ be such that $\tau^-_{\beta,\alpha}(1)=(10)^\infty$. In which case, $u(\beta,\alpha)$ is equal to the golden mean, which we denote by $G$, and belongs to the set $C$. Thus, $E^{+}_{\beta,\alpha}$ contains no isolated points.  From \cite[Proposition 5.2]{KKLL} and by an elementary calculation, we have that $t_{u(\beta,\alpha),c} = G^{-2}$ and $\tau_{G}^{-}(G^{-2}) = 00(10)^\infty$. This, in tandem with \Cref{prop:char(5)}, yields $t_{\beta,\alpha,c} = \pi_{\beta, \alpha} \tau_{G}^{-}(G^{-2}) = \pi_{\beta, \alpha}(00(10)^\infty) = \alpha/(1-\beta)+1/(\beta^{3}-1)$, which one can show is equal to $(1-\alpha-\beta\alpha)\beta^{-2}$ using the fact that $\tau^-_{\beta,\alpha}(1)=(10)^\infty$.  Moreover, by \Cref{prop:char} Part~(4),
    \begin{align*}
        \dim_H(K_{\beta,\alpha}(t))=(\log(G)/\log(\beta)) \dim_H(K_{G}(\tilde{\pi}_{\beta,\alpha}(t)),
    \end{align*}
for all $t \in (0, 1)$.  We now show that for a given $\beta\in(1,G)$ there exists a unique $\alpha\in(0,1/2)$ with $G = u(\beta,\alpha)$, or equivalently, that for a given $\beta\in(1,G)$ there exists a unique $\alpha\in(0,1/2)$ with $T_{\beta,\alpha}(1) = p_{\beta,\alpha}$. Using the definitions of the involved terms, $T_{\beta,\alpha}(1) = p_{\beta,\alpha}$ if and only if $\alpha=1-\beta^2/(\beta+1)$. Noting, when $\beta = G$, that $\alpha=0$, and as $\beta$ approaches $1$ from above, that $\alpha =1-\beta^2/(\beta+1)$ converges to $1/2$, yields the required result.

By \Cref{thm:LSSS}, under the assumption that $\tau^-_{\beta,\alpha}(1)=(10)^\infty$, if $\tau_{\beta,\alpha}^+(0)$ is periodic, then $\Omega_{\beta,\alpha}$ is a subshift of finite type. We now find $(\beta, \alpha) \in \Delta$ such that $\tau_{\beta,\alpha}^+(0)=(0001)^\infty$ and $\tau^-_{\beta,\alpha}(1)=(10)^\infty$. For this we, observe that
    \begin{align*}
        T_{\beta,\alpha}(0) = \alpha, \quad
        T_{\beta,\alpha}^2(0) = \beta\alpha+\alpha, \quad
        T_{\beta,\alpha}^3(0) = \beta(\beta\alpha+\alpha)+\alpha, \quad \text{and} \quad
        T_{\beta,\alpha}^4(0) = \beta(\beta(\beta\alpha+\alpha)+\alpha)+\alpha-1=0.
    \end{align*}
Substituting $\alpha=1-\beta^2/(\beta+1)$ into the last equality, gives
    \begin{align*}
        \beta(\beta(\beta(1-\beta^2/(\beta+1))+(1-\beta^2/(\beta+1)))+(1-\beta^2/(\beta+1)))+(1-\beta^2/(\beta+1))-1 = 0.
    \end{align*}
This reduces to $\beta(\beta^4 -\beta^2- \beta -1) =0$. Thus, if $\beta$ is the positive real root of $\beta^4 -\beta^2- \beta -1 =0$ and $\alpha=1-\beta^2/(\beta+1)$, then $\tau_{\beta,\alpha}^+(0)=(0001)^\infty$ and $\tau^-_{\beta,\alpha}(1)=(10)^\infty$.  Numerically approximating $\beta$ and $\alpha$ yields $\beta\approx 1.4656$ and $\alpha\approx 0.1288$.

We utilise the above, in particular \Cref{prop:char}, in studying the dimension function $\eta_{\beta,\alpha}$.  Recall, if $t\not \in E^{+}_{\beta,\alpha}$, then there exists $t^*>t$ with $K^{+}_{\beta,\alpha}(t)=K^{+}_{\beta,\alpha}(t^*)$.  Thus, it suffices to study $K^{+}_{\beta,\alpha}(t)$ for $t\in E^{+}_{\beta,\alpha}$. For a fixed $t\in E^{+}_{\beta,\alpha}$, with the aid of \Cref{thm:BSV14}, we find $(\beta^\prime,\alpha^\prime) \in \Delta$ with $\tau_{\beta,\alpha}(t)=\tau_{\beta^\prime,\alpha^\prime}(0)$ and  $\tau^{-}_{\beta,\alpha}(1) = \tau^{-}_{\beta^\prime,\alpha^\prime}(1)$. 
By \Cref{prop:char} Part~(4),
    \begin{align*}
        \eta_{\beta,\alpha}(t)=\frac{\log(u(\beta,\alpha))}{\log(\beta)} \dim_H(K_{u(\beta,\alpha)}^+(\tilde{\pi}_{\beta,\alpha}(t)))
        \quad \text{and} \quad
        \eta_{\beta^\prime,\alpha^\prime}(0)=\frac{\log(u(\beta^\prime,\alpha^\prime))}{\log(\beta^\prime)} \dim_H(K_{u(\beta^\prime,\alpha^\prime)}^+(\tilde{\pi}_{\beta^\prime,\alpha^\prime}(0))).
    \end{align*}
Since $u(\beta^\prime,\alpha^\prime)=u(\beta,\alpha)$, $\tilde{\pi}_{\beta,\alpha}(t)=\tilde{\pi}_{\beta^\prime,\alpha^\prime}(0)$ and $\eta_{\beta^\prime,\alpha^\prime}(0)=1$, we have $\eta_{\beta,\alpha}(t)=\log(\beta^\prime)/\log(\beta)$. In summary, determining the value of $\eta_{\beta, \alpha}(t)$ reduces down to finding such $\alpha^\prime$ and $\beta^\prime$. This can performed numerically with the aid of the monotonicity and continuity of the projection maps, see Figure \ref{fig:dim} for sample numerical outputs.

\section{Winning sets of intermediate \texorpdfstring{$\beta$}{beta}-transformations: Proof of Proof of \texorpdfstring{\Cref{thm:main_2}}{Theorem 1.6}}\label{sec:proof_thm_1_6}

To show the conditions of \Cref{thm_HY_Thm_2.1} are satisfied when $T=T_{\beta, \alpha}$ for all $(\beta, \alpha) \in \Delta$ with $T_{\beta, \alpha}$ transitive and $\Omega_{\beta, \alpha}$ of finite type we use the following lemma on the geometric length of cylinder sets and the following proposition.

\begin{lemma}\label{lem:geometric_lengths_of_cylinders}
Let $(\beta, \alpha) \in \Delta$ be such that $T_{\beta, \alpha}$ is transitive and $\Omega_{\beta, \alpha}$ is a subshift of finite type. If $\nu = \nu_{1} \cdots \nu_{\lvert \nu \rvert}$ is an admissible word with respect to the partition $P_{\beta,\alpha}$, then $\rho \beta^{-\lvert \nu \rvert} \leq  \lvert I(\nu) \rvert \leq \beta^{-\lvert \nu \rvert}$, where $\rho = \min \{ \lvert I(i) \rvert \colon i \in \Lambda \}$.
\end{lemma}

\begin{proof}
If $\lvert \nu \rvert = 1$, the result is a consequence of the fact that $\max\{ p_{\beta,\alpha}, 1 - p_{\beta,\alpha}\} \leq \beta^{-1}$, and that $I(\nu) \subseteq [0, p_{\beta,\alpha}]$ or $I(\nu) \subseteq [p_{\beta,\alpha},1]$.  Therefore, we may assume that $\lvert \nu \rvert \geq 2$.  Since $T_{\beta, \alpha}$ is Markov with respect to the partition $P_{\beta,\alpha}$, for $j \in \{ 0, 1, \ldots, \lvert \nu \rvert \}$, we have $T_{\beta,\alpha}^{j}(I(\nu))$ is an interval and  $T_{\beta,\alpha}^{\lvert \nu \rvert - 1}(J(\nu)) = J(\nu_{\lvert \nu \rvert})$, where for a given admissible finite word $\omega$, we denote by $J(\omega)$ the interior of $I(\omega)$. 
This implies that $\lvert I(\nu) \rvert = \beta^{-\lvert \nu \rvert + 1}\lvert I(\nu_{\lvert \nu \rvert}) \rvert$ and hence that $\rho \beta^{-\lvert \nu \rvert} \leq  \lvert I(\nu_{\lvert \nu \rvert}) \rvert \beta^{-\lvert \nu \rvert + 1} = \lvert I(\nu) \rvert = \lvert I(\nu_{\lvert \nu \rvert}) \rvert \beta^{-\lvert \nu \rvert + 1} \leq  \beta^{-\lvert \nu \rvert}$.
\end{proof}

\begin{proposition}\label{prop:thm_SFT+Transitive_implies_winning}
Under the hypotheses of \Cref{lem:geometric_lengths_of_cylinders}, for all $x \in [0, 1]$ and $\gamma \in (0, 1)$, we have that the geometric condition $H_{x, \gamma}$, with $T = T_{\beta, \alpha}$ and the partition $P_{\beta,\alpha}$, is satisfied.
\end{proposition}

\begin{proof}
\Cref{lem:geometric_lengths_of_cylinders} yields \eqref{eq:condition_1} of $H_{\xi, \gamma}$, thus is suffices to show that \eqref{eq:condition_2} of $H_{x, \gamma}$ is satisfied.  To this end, let $n-1$ denote the cardinality of $P_{\beta,\alpha}$, and observe that, since by assumption $T_{\beta, \alpha}$ is transitive, there exists an $m = m_{\beta, \alpha} \in \mathbb{N}$, so that $J(l) \subseteq T_{\beta,\alpha}^{m}(J(k))$, for all $l$ and $k \in \{0, 1, \ldots, n-1\}$, where $J(l)$ and $J(k)$ are as defined in the proof of \Cref{lem:geometric_lengths_of_cylinders}.  Further, if for two admissible words $\nu$ and $\eta$, we have that $I(\nu)$ and $I(\eta)$ are $\gamma/4$-comparable, then by \Cref{lem:geometric_lengths_of_cylinders} there exists $k_{0} \in \mathbb{N}$ with $\lvert \lvert \nu \rvert - \lvert \eta \rvert \rvert \leq k_{0}$.  Letting $\omega$ denote the symbolic representation of $x$ generated by $T_{\beta, \alpha}$ with respect to the partition $P_{\beta,\alpha}$, set
    \begin{align*}
        M = M_{\beta, \alpha} = \begin{cases}
        \max \{ \sigma^{k}(\omega) \wedge \omega \colon k \in \{1, 2, \ldots, k_{0} \}\} & \text{if} \; \omega \; \text{is not periodic},\\
        \max \{ \sigma^{k}(\omega) \wedge \omega \colon k \in \{1, 2, \ldots, \operatorname{per}(\omega) -1 \}\} & \text{if} \; \omega \; \text{is periodic.}
        \end{cases}
    \end{align*}
Our aim is to show \eqref{eq:condition_2} of $H_{x, \gamma}$ is satisfied for all admissible words $\nu$ and $\eta$ with $I(\nu)$ and $I(\eta)$ are $\gamma/4$-comparable and all integers $i > i^{*} = k_{0} + m + M$. For this, suppose $\nu$ and $\eta$ are $\omega\vert_{i}$-extendable with $0 \leq \lvert \lvert \nu \rvert - \lvert \eta \rvert \rvert \leq k_{0}$ and $\operatorname{dist}(I(\nu\omega\vert_{i}), I(\eta\omega\vert_{i})) > 0$.  We consider case when $\nu$ is not a prefix of $\eta$, and when $\nu$ is a prefix of $\eta$ separately.

For both of these cases we use the following facts. For $l \in \{ 0, 1, \ldots, n-1\}$ there exists a minimal $j \in \{ 1, 2, \ldots, m \}$ such that $T_{\beta, \alpha}^{j}(I(l))$ contains the interiors of at least two elements of $P_{\beta,\alpha}$.  For $k \in \{1, 2, \ldots, \lvert \nu \rvert + i -1\}$ and $l \in \{1, 2, \ldots, \lvert \eta \rvert + i -1\}$, we have $T_{\beta, \alpha}^{k}(I(\nu \omega\vert_{i}))$ and $T_{\beta, \alpha}^{l}(I(\eta \omega\vert_{i}))$ are intervals, and $T_{\beta, \alpha}^{k}(J(\nu \omega\vert_{i})) = J(\sigma^{k}(\nu \omega\vert_{i}))$ and $T_{\beta, \alpha}^{l}(J(\eta \omega\vert_{i})) = J(\sigma^{l}(\eta \omega\vert_{i}))$.

Let us consider the first of our two cases, namely when $\nu$ is not a prefix of $\eta$. Our above two facts imply that there exist $l \in \{1, 2, \ldots, m-1\}$ and $F \subseteq \{0, 1, \ldots, n-1\}$ with $\lvert F \rvert \geq 2$ such that
    \begin{align}\label{eq:HY-splitting_of_cylinders}
        I(\nu \omega\vert_{l}) = \bigcup_{k \in F} I(\nu \omega\vert_{l} k)
        \quad \text{and} \quad
        I(\eta \omega\vert_{l}) = \bigcup_{k \in F} I(\eta \omega\vert_{l} k).
    \end{align}
Since $\nu$ is not a prefix of $\eta$, there exists $j \in \{ 1, 2, \ldots, \min\{\lvert \nu \rvert, \lvert \eta \rvert \} - 1 \}$ such that $\nu\vert_{j} = \eta\vert_{j}$ and $\nu\vert_{j+1} \prec \eta\vert_{j+1}$, or $\nu\vert_{j} = \eta\vert_{j}$ and $\nu\vert_{j+1} \succ \eta\vert_{j+1}$.  Suppose that $\nu\vert_{j} = \eta\vert_{j}$ and $\nu\vert_{j+1} \prec \eta\vert_{j+1}$, and that $\omega_{1} = \max F$. Letting $k \in F \setminus \{\omega_{1}\}$, for all $x \in I(\nu \omega\vert_{i})$, $y \in I(\eta \omega\vert_{l} k)$ and $z \in I(\eta \omega\vert_{i})$, that $x \leq y \leq z$.  In other words, $\operatorname{dist}(I(\nu\omega\vert_{i}), I(\eta\omega\vert_{i})) \geq \lvert I(\eta \omega\vert_{l} k)\rvert$, and hence by \Cref{lem:geometric_lengths_of_cylinders},
    \begin{align*}
        \operatorname{dist}(I(\nu\omega\vert_{i}), I(\eta\omega\vert_{i}))
        \geq \lvert I(\eta \omega\vert_{l} k)\rvert
        \geq \rho \beta^{-(\lvert \eta \rvert + l + 1 )}
        \geq \rho \beta^{-(\lvert \eta \rvert + m + 1)}
        \geq \rho \beta^{-(m + 1)} \lvert I(\eta) \rvert.
    \end{align*}
Similarly, if $\omega_{1} \neq \max\{F\}$, setting $k = \max F$, we obtain that
    \begin{align*}
        \operatorname{dist}(I(\nu\omega\vert_{i}), I(\eta\omega\vert_{i}))
        \geq \lvert I(\nu \omega\vert_{l} k)\rvert
        \geq \rho \beta^{-(\lvert \nu \rvert + l + 1 )}
        \geq \rho \beta^{-(\lvert \nu \rvert + m + 1)}
        \geq \rho \beta^{-(m + 1)} \lvert I(\nu) \rvert.
    \end{align*}
An analogous argument yields the result when $\nu\vert_{j} = \eta\vert_{j}$ and $\nu\vert_{j+1} \succ \eta\vert_{j+1}$.

When $\nu$ is a prefix of $\eta$, the result follows using a similar reasoning as in the case  when $\nu$ is not a prefix of $\eta$, but where we replace the first line of the argument, namely \eqref{eq:HY-splitting_of_cylinders}, by the following observation. By construction, there exists a $j \in \{ 1, 2, \ldots, \lvert \eta \rvert - \lvert \nu \rvert + M -1 \}$ such that $\nu \omega\vert_{j} = (\eta \omega\vert_{i})\vert_{j + \lvert \nu \rvert}$ but $\nu \omega\vert_{j+1} \neq (\eta \omega\vert_{i})\vert_{j + \lvert \nu \rvert + 1}$.  In which case, by our two facts, there exists an $l \in \{0, 1, 2, \ldots, m-1\}$ and a subset of $F$ of $\{1, 2, \ldots, n-1\}$ with $\lvert F \rvert \geq 2$ such that
    \[
        I(\nu \omega\vert_{j+l}) = \bigcup_{k \in F} 
        I(\nu \omega\vert_{j+l} k)
        \quad \text{and} \quad
        I((\eta \omega\vert_{i})\vert_{j + \lvert \nu \rvert + l}) = \bigcup_{k \in F} I((\eta \omega\vert_{i})\vert_{j + \lvert \nu \rvert +l} k).
        \qedhere
    \]
\end{proof}

\begin{proof}[{Proof of \Cref{thm:main_2}}]
This is a direct consequence of \Cref{thm:G1990+Palmer79,thm_HY_Thm_2.1}, and \Cref{prop:alpha-winning_transport,prop:thm_SFT+Transitive_implies_winning}.
\end{proof}

\bibliographystyle{alpha}
\bibliography{alphawinning}

\end{document}